\documentclass[reqno,12pt]{amsart}

\usepackage[utf8]{inputenc}
\usepackage{microtype}
\usepackage{amsfonts}
\usepackage{amsmath}
\usepackage{graphicx}
\usepackage{float}
\usepackage{color}
\usepackage{hyperref}
\usepackage{tikz-cd}

\usepackage{amssymb}
\usepackage{enumitem}
\usepackage{mathrsfs}

\usepackage{tikz}
\usetikzlibrary{topaths}
\usetikzlibrary{calc}

\usepackage{a4wide}

\usepackage{empheq}

\newtheorem{theorem}{Theorem}[section]
\newtheorem*{theorem*}{Theorem}

\newtheorem{lemma}[theorem]{Lemma}

\newtheorem{proposition}[theorem]{Proposition}
\newtheorem*{proposition*}{Proposition}
\newtheorem{corollary}[theorem]{Corollary}
\newtheorem*{corollary*}{Corollary}

\newtheorem{cit}[theorem]{Citation}
\newtheorem*{conjecture*}{Conjecture}

\newtheorem*{question*}{Question}

\newtheorem*{main:braided_nekr_inherit_finprops}{Theorem~\ref{thrm:braided_nekr_inherit_finprops}}
\newtheorem*{main:braided_roever_F_infty}{Theorem~\ref{thrm:braided_roever_F_infty}}

\theoremstyle{definition}
\newtheorem{definition}[theorem]{Definition}
\newtheorem{remark}[theorem]{Remark}
\newtheorem{example}[theorem]{Example}

\newcommand{\Z}{\mathbb{Z}}
\newcommand{\N}{\mathbb{N}}

\newcommand{\R}{\mathbb{R}}

\newcommand{\lk}{\operatorname{lk}}
\newcommand{\st}{\operatorname{st}}

\newcommand{\dlk}{\lk^\downarrow\!}
\newcommand{\dst}{\st^\downarrow\!}

\newcommand{\clone}{\kappa}
\newcommand{\symmclone}{\varsigma}
\newcommand{\braidclone}{\vartheta}
\newcommand{\tree}{\mathcal{T}}
\newcommand{\states}{\psi}
\newcommand{\splitting}{\Delta}
\newcommand{\altsplitting}{\Omega}
\newcommand{\merging}{\Upsilon}
\newcommand{\altmerging}{\Psi}
\newcommand{\disk}{D}
\newcommand{\surface}{\mathcal{S}}

\newcommand{\defeq}{\mathbin{\vcentcolon =}}

\DeclareMathOperator{\Aut}{Aut}
\DeclareMathOperator{\AAut}{AAut}

\DeclareMathOperator{\Symm}{Symm}
\DeclareMathOperator{\Grig}{Grig}
\DeclareMathOperator{\brGrig}{brGrig}
\DeclareMathOperator{\poly}{poly}

\DeclareMathOperator{\F}{F}

\DeclareMathOperator{\Stab}{Stab}

\DeclareMathOperator{\id}{id}

 \newcommand{\diskcpx}{{\mathbb {D}}}
 \newcommand{\CM}{{\mathcal {M}}} 

\newcommand{\Thomp}{\mathscr{T}}   

\newcommand{\Groupoid}{\mathscr{G}}
\newcommand{\Poset}{\mathscr{P}}
\newcommand{\Stein}{\mathscr{X}}
\newcommand{\dlmodel}{\mathscr{L}}

\newcommand{\brF}%
   {\operatorname{br}\!F}                 

\newcommand{\brT}%
   {\operatorname{br}\!T}                 

\newcommand{\brV}%
   {\operatorname{br}\!V}                 
   
\newcommand{\brAut}%
   {\operatorname{br}\!\Aut}                 
   
\newcommand{\brAAut}%
   {\operatorname{br}\!\AAut}                 
   
\newcommand{\brW}%
   {\operatorname{br}\!W}                 
   
\newcommand{\brR}%
   {\operatorname{br}\!R}
   
\numberwithin{equation}{section}

\begin{document}

\title{Braiding groups of automorphisms and almost-automorphisms of trees}
\date{\today}


\author[R.~Skipper]{Rachel Skipper}
\address{D\'epartement de math\'ematiques et applications, \'Ecole normale sup\'erieure, 75005 Paris, France}
\email{rachel.skipper@ens.fr}

\author[M.~C.~B.~Zaremsky]{Matthew C.~B.~Zaremsky}
\address{Department of Mathematics and Statistics, University at Albany (SUNY), Albany, NY 12222}
\email{mzaremsky@albany.edu}

\begin{abstract}
We introduce ``braided'' versions of self-similar groups and R\"over--Nekrashevych groups, and study their finiteness properties. This generalizes work of Aroca and Cumplido, and the first author and Wu, who considered the case when the self-similar groups are what we call ``self-identical''. In particular we use a braided version of the Grigorchuk group to construct a new group called the braided R\"over group, which we prove is of type $\F_\infty$. Our techniques involve using so called $d$-ary cloning systems to construct the groups, and analyzing certain complexes of embedded disks in a surface to understand their finiteness properties.
\end{abstract}

\maketitle
\thispagestyle{empty}

\section*{Introduction}

In this paper we introduce and study braided versions of self-similar groups of tree automorphisms and R\"over--Nekrashevych groups of tree almost-automorphisms. A subgroup of the group of automorphisms of an infinite rooted regular tree is called self-similar if it is ``built out of copies of itself'' in some sense (see Definition~\ref{def:self_sim}). The R\"over--Nekrashevych group associated to a self-similar group is the group of self-homeomorphisms of the boundary of the tree that locally ``look like'' the self-similar group (see Definition~\ref{def:nekr}). The first example, now called the R\"over group, was constructed by R\"over in \cite{roever99}, using the Grigorchuk group constructed by Grigorchuk in \cite{grigorchuk80,grigorchuk84}. In \cite{nekrashevych04} Nekrashevych constructed R\"over--Nekrashevych groups in general, starting with an arbitrary self-similar group.

The braided variants of R\"over--Nekrashevych groups we construct here were previously considered by Aroca--Cumplido \cite{aroca22} in the special case when the self-similar groups are what we call ``self-identical'' (Definition~\ref{def:self_identical}), and these examples were also studied by the first author and Wu in \cite{skipperwu}. These papers are  part of a large body of recent work devoted to ``braiding'' groups in the extended family of Thompson's groups. The original braided Thompson group $\brV$ was introduced independently by Brin \cite{brin07} and Dehornoy \cite{dehornoy06}, as a braided version of the classical Thompson group $V$. Thompson's group $F$ also has a braided counterpart $\brF$, first considered in \cite{brady08}. Thompson's group $T$ has a number of different ``braidings'', found, e.g., in \cite{funar08,funar11,witzel19}. Braided versions of the Brin--Thompson groups from \cite{brin04} were recently constructed by Spahn \cite{spahn}. Finally, Houghton's groups, which are related to Thompson's groups, have braided variants due to Degenhardt \cite{degenhardt00}. The braided Houghton groups along with the braided versions of $T$ from \cite{funar08,funar11} are also studied in \cite{genevois22}.

In this paper we ``braid'' arbitrary self-similar groups and arbitrary R\"over--Nekrashevych groups. In particular we get a new group that we call the \emph{braided R\"over group}, which is a braided variant of the R\"over group, i.e., the R\"over--Nekrashevych group specifically arising from the Grigorchuk group. (A torsion-free version of the Grigorchuk group that could be called a ``braided'' version was previously constructed by Grigorchuk in \cite{grigorchuk85}; the braided Grigorchuk group we construct here is slightly different.) To construct braided self-similar groups, we work inside an infinitely iterated wreath product of braid groups, and impose conditions arising from a ``braided wreath recursion'' on elements (see Definition~\ref{def:braided_self_sim}). To construct braided R\"over--Nekrashevych groups, we use the framework of $d$-ary cloning systems, developed by the authors and Stefan Witzel in \cite{witzel18,skipper21}: we use the braided wreath recursion to produce a $d$-ary cloning system associated to any braided self-similar group, and then the braided R\"over--Nekrashevych group is the resulting Thompson-like group that arises from the $d$-ary cloning system (see Definition~\ref{def:braided_ro_ne}).

As often happens with new Thompson-like groups, for example in many of the aforementioned references, it is of interest to understand their finiteness properties. A group is of \emph{type $F_n$} if it has a classifying space with finite $n$-skeleton, where a \emph{classifying space} is a connected CW-complex whose fundamental group is the group in question and whose higher homotopy groups are all trivial. Thus, type $\F_1$ is equivalent to finite generation and type $\F_2$ is equivalent to finite presentability. We say \emph{type $F_\infty$} for type $\F_n$ for all $n$. The classical Thompson's groups $F$, $T$, and $V$, and the Brin--Thompson groups $nV$, are all of type $\F_\infty$ \cite{brown84,brown87,fluch13}, as are all their aforementioned braided variants \cite{bux16,witzel19,genevois22,spahn}. There is also a ``ribbon braided'' variant of $V$ that is of type $\F_\infty$ \cite{thumann17}. The R\"over group along with some related R\"over--Nekrashevych groups are of type $\F_\infty$ as well \cite{farley15,belk16,skipper21}. In particular this includes R\"over--Nekrashevych groups arising from self-identical groups \cite{farley15}. R\"over--Nekrashevych groups also afforded the first known examples of simple groups of type $\F_{n-1}$ but not $\F_n$ for arbitrary $n$, in \cite{skipper19}. Here we analyze finiteness properties of braided R\"over--Nekrashevych groups, and prove two main results:

\begin{main:braided_nekr_inherit_finprops}
Let $G\le \brAut(\tree_d)$ be braided self-similar. If $G$ is of type $\F_n$ then so is $\brV_d(G)$.
\end{main:braided_nekr_inherit_finprops}

\begin{main:braided_roever_F_infty}
The braided R\"over group is of type $\F_\infty$.
\end{main:braided_roever_F_infty}

When $G$ is braided self-identical, Theorem~\ref{thrm:braided_nekr_inherit_finprops} was proved by the first author and Wu in \cite{skipperwu}, where it was proved that the converse also holds. We prove that the braided Grigorchuk group is (like the original Grigorchuk group) not finitely presentable (Proposition~\ref{prop:brGrig_not_fp}), so Theorem~\ref{thrm:braided_roever_F_infty} shows that the converse of Theorem~\ref{thrm:braided_nekr_inherit_finprops} is not true in general. We remark that it remains an interesting problem for $n\ge 2$ to find explicit examples of $G$ for which Theorem~\ref{thrm:braided_nekr_inherit_finprops} applies but the results in \cite{skipperwu} do not, i.e., $G$ is braided self-similar but not braided self-identical. (For an example when $n=1$, see Example~\ref{ex:ZwrZ} using $\Z\wr\Z$.) In another direction, for $G$ a braided self-similar group and $\pi(G)$ its corresponding (non-braided) self-similar group, it would be interesting to try and relate the finiteness properties of $G$ to those of $\pi(G)$, and the finiteness properties of $\brV_d(G)$ to those of $V_d(\pi(G))$; we leave this for future investigation.

This paper also rectifies a gap in the literature, namely we construct the so called Stein--Farley complex for a general $d$-ary cloning system. When $d=2$ this was done by Witzel and the second author in the original cloning systems paper \cite{witzel18}, but for general $d$-ary cloning systems, introduced in \cite{skipper21}, the Stein--Farley complex was only constructed for a special case involving self-similar groups. Here we officially construct the Stein--Farley complex for an arbitrary $d$-ary cloning system (Subsection~\ref{ssec:stein_farley}). The construction is straightforward, and works essentially by combining the ideas from the two aforementioned special cases, but had not technically been done before. These complexes are a key tool in proving Theorem~\ref{thrm:braided_nekr_inherit_finprops}.

In order to prove Theorem~\ref{thrm:braided_roever_F_infty}, we introduce a new type of complex defined on a surface, which we call the $(2,5/2)$-disk complex (see Definition~\ref{def:252}). A vertex of this complex is an isotopy class of an embedded disk, enclosing $2$ marked points in its interior and either $0$ or $1$ marked points in its boundary, and a collection of vertices span a simplex whenever the disks are pairwise disjoint or nested. The higher connectivity properties of the descending links in the Stein--Farley complex for the braided R\"over group turn out to be informed by those of the $(2,5/2)$-disk complex, and this complex seems to be of interest in its own right, given its connection to (braided versions of) the Grigorchuk and R\"over groups.

This paper is organized as follows. In Section~\ref{sec:auts} we recall the background on self-similar groups, and define braided self-similar groups. In Section~\ref{sec:aauts} we recall the background on $d$-ary cloning systems and R\"over--Nekrashevych groups, and define braided R\"over--Nekrashevych groups, including the braided R\"over group. In Section~\ref{sec:fin_props} we construct Stein--Farley complexes for arbitrary $d$-ary cloning systems (which has technically not been done before), and then focus on the case of braided R\"over--Nekrashevych groups to prove Theorem~\ref{thrm:braided_nekr_inherit_finprops}. Finally, in Section~\ref{sec:braided_roever_fin_props} we prove Theorem~\ref{thrm:braided_roever_F_infty}, that the braided R\"over group is of type $\F_\infty$.

\subsection*{Acknowledgments} We are grateful to Daniel Allcock, Jim Belk, and Xiaolei Wu for helpful discussions, to Anthony Genevois for pointing out the reference \cite{benli13}, and to the anonymous referee for some excellent suggestions. The first author is supported by NSF DMS--2005297 and the European Research Council (ERC) under the European Union’s Horizon 2020 research and innovation program (grant agreement No.725773). The second author is supported by grant \#635763 from the Simons Foundation.

\section{Braiding groups of automorphisms of trees}\label{sec:auts}

Let $d\in \N$ with $d\ge 2$, let $X$ be a set with $d$ elements, called an \emph{alphabet}, and let $X^*$ be the set of all finite words in $X$ (including the empty word, denoted $\varnothing$). The infinite rooted $d$-ary tree, denoted $\tree_d$, naturally has $X^*$ as its vertex set. The root is $\varnothing$, and given a vertex $v$ the \emph{children} of $v$ are the vertices of the form $vx$ for $x\in X$.

\begin{definition}[Automorphism]
An \emph{automorphism} of $\tree_d$ is a bijection $X^*\to X^*$ that preserves incidence (and so in particular fixes $\varnothing$). The group of all automorphisms of $\tree_d$ is denoted $\Aut(\tree_d)$.
\end{definition}

We will be interested in certain subgroups of $\Aut(\tree_d)$, called self-similar groups. A vast amount of information about self-similar groups can be found in \cite{nekrashevych05}. To define them we need to view $\Aut(\tree_d)$ as an infinitely iterated wreath product, namely
\[
\Aut(\tree_d) \cong S_d \wr_X(S_d \wr_X(S_d \wr_X\cdots)) \text{,}
\]
as we will now begin to explain.

Our convention for wreath products is that $S_d \wr_X G \defeq S_d \ltimes G^X$, i.e., the group doing the acting is written on the left. We write the subscript $X$ in $\wr_X$ to emphasize that this is the permutation wreath product coming from $S_d\defeq \Symm(X)$ acting on $X$. We may also sometimes identify $X$ with $\{1,\dots,d\}$ so that elements of $G^X$ can be conveniently written as tuples $(g_1,\dots,g_d)$. We will also sometimes write $S_d \wr_X^\infty S_d$ for $S_d \wr_X(S_d \wr_X(S_d \wr_X\cdots))$.

Now let us be more rigorous about infinitely iterated wreath products. Let $S_d \wr_X^n S_d$ be the $n$-times-iterated wreath product, e.g.,
\begin{align*}
S_d\wr_X^0 S_d &= S_d\\
S_d \wr_X^1 S_d &= S_d \wr_X S_d\\
S_d \wr_X^2 S_d &= S_d \wr_X (S_d \wr_X S_d) \text{,}
\end{align*}
and so forth. For each $n\in\N$ we have an epimorphism $S_d \wr_X^n S_d \to S_d \wr_X^{n-1} S_d$ given by ``forgetting the rightmost factor''; for example $S_d \wr_X^2 S_d \to S_d \wr_X^1 S_d$ is the map
\[
(\sigma,((\tau_1,(\upsilon_1^1,\dots,\upsilon_1^d)),\dots,(\tau_d,(\upsilon_d^1,\dots,\upsilon_d^d)))) \mapsto (\sigma,(\tau_1,\dots,\tau_d)) \text{.}
\]
This forms a projective system, and the infinitely iterated wreath product $S_d \wr_X^\infty S_d$ is the projective limit of this system. It is clear that this is isomorphic to $\Aut(\tree_d)$.  The point is that an automorphism of $\tree_d$ can be viewed as first shuffling the $d$ many children of the root, then independently for each child of the root shuffling its $d$ many children, then shuffling their children, and so on.

Thanks to this viewpoint we see that $\Aut(\tree_d)\cong S_d\wr_X\Aut(\tree_d)$. Hence, an element $f\in \Aut(\tree_d)$ can be decomposed as $f=\rho(f)(f_1,\dots,f_d)$, where
\[
\rho\colon \Aut(\tree_d)\to S_d
\]
is the natural epimorphism coming from the wreath product, and $f_i\in\Aut(\tree_d)$. This is called the \emph{wreath recursion}. This produces a function (not a homomorphism) $\states\colon \Aut(\tree_d)\to \Aut(\tree_d)^X$ sending $f$ to
\[
\states(f)\defeq (f_1,\dots,f_d)\text{,}
\]
so the wreath recursion of $f$ is $f=\rho(f)\states(f)$. For any subgroup $G\le\Aut(\tree_d)$ the image $\states(G)$ is a subset of $\Aut(\tree_d)^X$, and this leads us to the definition of self-similar:

\begin{definition}[Self-similar]\label{def:self_sim}
A subgroup $G\le \Aut(\tree_d)$ is called \emph{self-similar} if $\states(G)\subseteq G^X$.
\end{definition}

The easiest example of a self-similar subgroup of $\Aut(\tree_d)$ is given by a certain action of the symmetric group $S_d$ on $\tree_d$. This kind of example will come up a lot, so we will actually give it its own name:

\begin{definition}[Self-identical]\label{def:self_identical}
A subgroup $G\le \Aut(\tree_d)$ is called \emph{self-identical} if
\[
\states(g)=(g,\dots,g)
\]
for all $g\in G$.
\end{definition}

\begin{example}[Subgroups of the symmetric group]\label{ex:symm_self_sim}
For $\sigma\in S_d=\Symm(X)$ we can define an element of $\Aut(\tree_d)$, also denoted $\sigma$, by declaring that $\sigma$ sends the vertex $v=x_1\cdots x_k$ to the vertex $\sigma(v)\defeq \sigma(x_1)\cdots\sigma(x_k)$. Note that the wreath recursion of $\sigma$ is $\sigma=\sigma(\sigma,\dots,\sigma)$. In particular $\states(\sigma)=(\sigma,\dots,\sigma)$, and so viewing $S_d$ as a subgroup of $\Aut(\tree_d)$ in this way we see that $S_d$, and indeed any subgroup of $S_d$, is self-identical (hence self-similar). Every self-identical subgroup occurs in this way.
\end{example}

\begin{example}[Grigorchuk group]
The \emph{Grigorchuk group}, introduced by Grigorchuk in \cite{grigorchuk80}, is the subgroup $\Grig\le\Aut(\tree_2)$ generated by elements $\overline{a}$, $\overline{b}$, $\overline{c}$, and $\overline{d}$ defined by the following wreath recursions. (These are usually denoted by $a$, $b$, $c$, and $d$, but we will be writing those for the braided version, so we will write $\overline{a}$, $\overline{b}$, $\overline{c}$, and $\overline{d}$ here.) First, $\overline{a}=(1~2)(\id,\id)$, where we identify $X$ with $\{1,2\}$. Next, $\overline{b}=(\overline{a},\overline{c})$, $\overline{c}=(\overline{a},\overline{d})$, and $\overline{d}=(\id,\overline{b})$, where the lack of a symbol in front of the ordered pair indicates that $\overline{b}$, $\overline{c}$, and $\overline{d}$ fix the children of the root. The Grigorchuk group was the first example of a finitely generated group that has intermediate growth, and that is amenable but not elementary amenable \cite{grigorchuk84}.
\end{example}

\subsection{Braiding groups of automorphisms}\label{ssec:braid_aut}

Now we describe a braided version of all of the above. Let $B_d$ be the $d$-strand braid group. Via the standard projection $B_d\to S_d$, we get an action of $B_d$ on $X$. Hence we can consider finitely iterated wreath products $B_d\wr_X^n B_d$, take the projective limit, and get the infinitely iterated wreath product
\[
B_d\wr_X (B_d \wr_X(B_d \wr_X\cdots)) \text{,}
\]
which we may also write as $B_d \wr_X^\infty B_d$.

\begin{definition}[Braided $\Aut(\tree_d)$]
Call the above infinitely iterated wreath product the \emph{braided automorphism group of $\tree_d$}, denoted
\[
\brAut(\tree_d) \defeq B_d \wr_X^\infty B_d \text{.}
\]
\end{definition}

Note that $\brAut(\tree_d) \cong B_d\wr_X \brAut(\tree_d)$, and so just like with $\Aut(\tree_d)$ we get a ``braided wreath recursion'': any $f\in \brAut(\tree_d)$ decomposes as $f=\phi(f)(f_1,\dots,f_d)$, where
\[
\phi\colon \brAut(\tree_d)\to B_d
\]
is the natural epimorphism coming from the wreath product, and $f_i\in \brAut(\tree_d)$. Also note that the epimorphism $B_d\to S_d$ induces an epimorphism
\[
\pi\colon \brAut(\tree_d)\to \Aut(\tree_d)\text{.}
\]
For any $f\in\brAut(\tree_d)$ define
\[
\states(f)\defeq(f_1,\dots,f_d) \text{,}
\]
where $f=\phi(f)(f_1,\dots,f_d)$ is the braided wreath recursion. In particular for any subgroup $G\le \brAut(\tree_d)$ the image $\states(G)$ is a subset of $\brAut(\tree_d)^X$. This leads us to the following:

\begin{definition}[Braided self-similar]\label{def:braided_self_sim}
A subgroup $G\le \brAut(\tree_d)$ is called \emph{braided self-similar} if $\states(G)\subseteq G^X$.
\end{definition}

\begin{definition}[Braided self-identical]
A subgroup $G\le \brAut(\tree_d)$ is called \emph{braided self-identical} if $\states(g)=(g,\dots,g)$ for all $g\in G$.
\end{definition}

Note that the image in $\Aut(\tree_d)$ under $\pi$ of any braided self-similar group is a self-similar group (with an analogous statement for self-identical).

\begin{remark}
In terms of being able to construct these variants of self-similar groups, there is nothing particularly special about braid groups. More generally, given any group $\Gamma$ acting on any set $Y$ one can form the iterated wreath products $\Gamma \wr_Y^n \Gamma$, take the projective limit to get $\Gamma \wr_Y (\Gamma \wr_Y(\Gamma \wr_Y\cdots))$, and then consider subgroups $G$ for which the associated wreath recursion of any element of $G$ involves only elements of $G$. One could call these ``$(\Gamma,Y)$-self-similar''. Our focus here is on braided versions, so we will not pursue this degree of generality here.
\end{remark}

\begin{example}[Subgroups of the braid group]\label{ex:braid_self_sim}
This example is the ``braided'' version of Example~\ref{ex:symm_self_sim}. For $\beta\in B_d$ we can define an element of $\brAut(\tree_d)$, also denoted $\beta$, via the braided wreath recursion $\beta=\beta(\beta,\dots,\beta)$. In particular, viewing $B_d$ as a subgroup of $\brAut(\tree_d)$ in this way we see that $B_d$, and indeed any subgroup of $B_d$, is braided self-identical (hence braided self-similar).
\end{example}

Before the next example, let us fix a generator $\zeta$ of $B_2\cong\Z$, which will also be useful in all that follows. We will use the braid where the strand on the left crosses over the strand on the right as the strands go down, as in Figure~\ref{fig:zeta}.

\begin{figure}[htb]
\centering
\begin{tikzpicture}[line width=0.8pt]
  \draw (1,0) to [out=-90, in=90] (0,-2);
  \draw[white, line width=4pt] (0,0) to [out=-90, in=90] (1,-2);
  \draw (0,0) to [out=-90, in=90] (1,-2);
\end{tikzpicture}
\caption{The braid $\zeta$.}
\label{fig:zeta}
\end{figure}
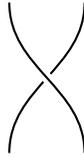

\begin{example}[$\Z\wr\Z$]\label{ex:ZwrZ}
As a nice example that is braided self-similar but not braided self-identical, we can use $\Z\wr\Z$. Take elements $a$ and $b$ whose braided wreath recursions are $a=\zeta(1,a)$ and $b=\zeta^2(1,b)$. A computation shows that $a^{2k}ba^{-2k}=\zeta^2(1,a^k ba^{-k})$ and $a^{2k-1}ba^{-(2k-1)}=\zeta^2(a^k ba^{-k},1)$ for all $k\in\Z$, so an induction argument shows that all the $b_n\defeq a^n ba^{-n}$ pairwise commute. Hence we get a well defined epimorphism $\Z\wr\Z\to \langle a,b\rangle$, and we claim it is an isomorphism. It suffices to show that if $(b_{i_1})^{p_1}\cdots (b_{i_\ell})^{p_\ell} a^q=1$ for $i_1<\cdots<i_\ell$ then $p_1=\cdots=p_\ell=q=0$. Hitting $\langle a,b\rangle$ with $\pi$, since $b\in\ker(\pi)$ and one can check that $a$ maps under $\pi$ to an infinite order element of $\Aut(\tree_2)$, we see that $q=0$. Now applying braided wreath recursions and using induction on the $p_i$, we see that the only possibility is $p_1=\cdots=p_\ell=0$, as desired.
\end{example}

Before discussing our main example of the braided Grigorchuk group, let us pin down the kernel of the map $\pi\colon \brAut(\tree_d)\to \Aut(\tree_d)$, i.e., $\pi\colon B_d\wr_X^\infty B_d\to S_d\wr_X^\infty S_d$. First note that the kernel of $\pi\colon B_d\to S_d$ is the \emph{pure braid group} $PB_d$. The action of $PB_d$ on $X$ is trivial, so $PB_d \wr_X^n PB_d$ is simply a direct product of copies of $PB_d$, namely $1+d+d^2+\cdots+d^n$ many copies. This equals the kernel of the natural map $B_d \wr_X^n B_d \to S_d \wr_X^n S_d$. Now we get the following:

\begin{lemma}\label{lem:kernel_is_pure}
The kernel of $\pi\colon B_d\wr_X^\infty B_d\to S_d\wr_X^\infty S_d$ is $PB_d\wr_X^\infty PB_d$, which is a direct product of infinitely many copies of $PB_d$.
\end{lemma}

\begin{proof}
Viewing the projective limit $B_d\wr_X^\infty B_d$ as a subgroup of the direct product of the factors $B_d \wr_X^n B_d$ in the projective system, and similarly $S_d\wr_X^\infty S_d$ as a subgroup of the direct product of the $S_d \wr_X^n S_d$, the map $\pi$ is the restriction of the analogous map between these direct products. The kernel of the map on the direct products is clearly the direct product of the $PB_d \wr_X^n PB_d$. This shows that the kernel of $\pi$ is the intersection of $B_d\wr_X^\infty B_d$ with the direct product of the $PB_d \wr_X^n PB_d$, which is $PB_d\wr_X^\infty PB_d$. Since $PB_d$ acts trivially on $X$, this is just a direct product of copies of $PB_d$.
\end{proof}

\begin{corollary}\label{cor:kernel_is_abelian}
The kernel of $\pi\colon \brAut(\tree_2)\to \Aut(\tree_2)$ is abelian.
\end{corollary}

\begin{proof}
By Lemma~\ref{lem:kernel_is_pure}, the kernel is a direct product of copies of $PB_2\cong \Z$.
\end{proof}

\subsection{The braided Grigorchuk group}\label{ssec:brGrig}

Now we construct a braided version of $\Grig$. Denote by $a$ the element of $\brAut(\tree_2)$ defined by the braided wreath recursion $a=\zeta(1,1)$. Now define $b$, $c$, and $d$ via the braided wreath recursions $b=(a,c)$, $c=(a^{-1},d)$, and $d=(1,b)$.

\begin{definition}[Braided Grigorchuk group]
The braided Grigorchuk group $\brGrig$ is the subgroup of $\brAut(\tree_2)$ defined by
\[
\brGrig\defeq \langle a,b,c,d\rangle\text{.}
\]
\end{definition}

The reason for using $c=(a^{-1},d)$ rather than $c=(a,d)$ is so that we get the pleasant-looking relation $bcd=1$ that holds analogously in the Grigorchuk group, as the proof of the next result shows.

\begin{lemma}\label{lem:Z2_in_brGrig}
The subgroup $\langle b,c,d\rangle$ of $\brGrig$ is isomorphic to $\Z^2$.
\end{lemma}

\begin{proof}
First we claim that $bcd=1$. Applying braided wreath recursions to $bcd$ produces $(1,cdb)$, then $(1,(1,dbc))$, and so forth, which shows that indeed $bcd=1$. A similar argument shows $cbd=1$. In particular $b$ and $c$ commute and $d=(bc)^{-1}$. It remains to show that $b$ and $c$ admit no non-trivial relations of the form $b^k c^\ell=1$. Applying braided wreath recursions we get $b^k c^\ell = (a^{k+\ell},(a^k,d^k b^\ell))$, which equals $1$ only if $k+\ell=0$ and $k=0$, so indeed no non-trivial such relations hold.
\end{proof}

\begin{lemma}\label{lem:brGrig_to_Grig}
The restriction of $\pi\colon \brAut(\tree_2)\to \Aut(\tree_2)$ to $\brGrig$ yields an epimorphism from $\brGrig$ onto $\Grig$.
\end{lemma}

\begin{proof}
First note that $\pi\colon B_2\to S_2$ sends $\zeta$ to $(1~2)$. Thus $\pi\colon \brAut(\tree_2)\to \Aut(\tree_2)$ sends $a$ to the element $\overline{a}$ of $\Grig$. It follows immediately that $b$, $c$, and $d$ respectively map under $\pi$ to $\overline{b}$, $\overline{c}$, and $\overline{d}$.
\end{proof}

\begin{corollary}\label{cor:brGrig_amenable}
The braided Grigorchuk group $\brGrig$ is amenable.
\end{corollary}

\begin{proof}
By Lemma~\ref{lem:brGrig_to_Grig} we have an epimorphism $\pi\colon \brGrig\to \Grig$. The kernel is abelian by Corollary~\ref{cor:kernel_is_abelian}. Hence $\brGrig$ is abelian-by-amenable, so amenable.
\end{proof}

Momentarily, we will prove that the braided Grigorchuk group is not finitely presented. The proof is inspired by the proof for the Grigorchuk group with some modifications, and we (roughly) follow this proof as given in \cite{dlHar00}. First we need some setup.

By Lemma~\ref{lem:Z2_in_brGrig}, we see there is a canonical epimorphism from $F\defeq \Z \ast \Z^2$ onto $\brGrig$ by mapping the generator of the first copy of $\Z$ to $a$ and mapping the generators of $\Z^2$ to $b$ and $c$. Thus any element in $\brGrig$ can be written (non-uniquely) as 
\[
z_1a^{k_1}z_2\cdots  a^{k_\ell}z_{\ell+1}
\]
where the $z_i$ are of the form $b^{m_i}c^{n_i}$ for some $m_i, n_i \in \Z$ and are nontrivial except possibly when $i=1$ or $\ell+1$. Call an expression of this form \emph{reduced}. For a reduced expression as above, declare the \emph{length}, denoted $|z_1a^{k_1}z_2\cdots  a^{k_\ell}z_{\ell+1}|$, to be the number of terms in the alternating product, so the length is  $2\ell-1$, $2\ell$, or $2\ell+1$ depending on whether $z_1$ and/or $z_{\ell+1}$ are equal to $1$. Similarly, define $|z_1a^{k_1}z_2\cdots  a^{k_\ell}z_{\ell+1}|_a=\sum_{i=1}^\ell k_i$, i.e., the sum of the exponents on the $a$ terms. 

Given any word $w$ in the generators $a$, $b$, $c$, and $d$ and their inverses, one can obtain a new word in reduced form by iteratively applying the following reductions and the analogous ones for their inverses:
\begin{enumerate}
    \item $g^ig^j=g^{i+j}$ for $g\in \{a,b, c, d\}$ 
    \item $d^{-1}=bc$
    \item $cb=bc$
\end{enumerate}
Call the resulting word $w^{red}$.

\begin{lemma}\label{lem:br_grig_contract}
Let $w$ be a reduced expression for an element of $\brGrig$. Consider the braided wreath recursion $w=\phi(w)(w_1, w_2)$. Then \[|w_i^{red}|\leq \Big\lfloor \frac{|w|+1}{2} \Big\rfloor \text{. }\]
\end{lemma}

\begin{proof}
Using the braided wreath recursion and applying the reductions we get
\[(b^kc^m)a^n= \begin{cases} 
      \zeta^n(a^{k-m}, b^{-m}c^{k-m}) & \text{ if } n \text{ is even} \\
      \zeta^n(b^{-m}c^{k-m}, a^{k-m}) & \text{ if } n \text{ is odd.}
   \end{cases} 
\]
Thus each pair $z_i a^{k_i}$ in the alternating product contributes at most one term to $w_1$ and at most one term to $w_2$. The result now follows.
\end{proof}

\begin{lemma}\label{lem:wordproblem}
The braided Grigorchuk group has solvable word problem.
\end{lemma}

\begin{proof}
Let $w$ be an expression in reduced form. To decide if $w$ represents the identity, proceed as follows. 
\begin{enumerate}
    \item Determine if $|w|_a= 0$. 
    \begin{enumerate}
        \item If $|w|_a \ne 0$, then $w\ne 1$
        \item If $|w|_a=0$ and the length of $w$ is $0$, then $w=1$. 
        \item If $|w|_a=0$ and the length of $w$ is at least 1, then apply $\states$ to obtain $\states(w)=(w_1, w_2)$ and go to $(ii).$
    \end{enumerate}
\item Compute $w_1^{red}$ and $w_2^{red}$ and return to $(i)$ which should be checked for both $w_1^{red}$ and $w_2^{red}$.
\end{enumerate}
It follows from Lemma~\ref{lem:br_grig_contract} that the procedure terminates.
\end{proof}

Let $W^{red}$ be the set of reduced words, so we can identify $W^{red}$ with $F=\Z \ast \Z^2$. For $w\in W^{red}$ and for $j_1, \dots, j_n\in \{1,2\}$, let $w_{j_1\cdots j_n}^{red}$ denote the reduced word defined inductively by
\[
w_{j_1\cdots j_n}^{red}=((w_{j_1\cdots j_{n-1}})^{red})_{j_n}^{red} \text{.}
\]
Let $\Psi\colon F\rightarrow \brGrig$ be the canonical epimorphism.

\begin{definition}
For each $n\geq 0$, let $K_n$ denote the subset of $F$ given by
\[
K_n\defeq \{w\in F \mid w\in\ker(\Psi), w_{j_1\cdots j_n}^{red}=1 \text{ for all } j_1, \dots, j_n\in\{1,2\} \}.
\]
\end{definition}

\begin{lemma}\label{lem:proper}
We have $\{1\}=K_0\leq K_1 \leq K_2\leq \cdots \leq \bigcup\limits_{n=0}^\infty K_n=\ker{\Psi}$. Moreover, each $K_n$ is normal in $F$ and all the inclusions are strict.
\end{lemma}

\begin{proof}
It is clear that $K_n\leq K_{n+1}$. The fact that $\bigcup\limits_{n=0}^\infty K_n=\ker{\Psi}$ follows from the fact that the procedure in Lemma~\ref{lem:wordproblem} terminates. It remains to show that each $K_n$ is normal and that the inclusions are strict.

It is clear that $K_0$ is normal in $F$ and so we proceed by induction. Observe that 
\[K_n=\{w\in F \mid |w|_a=0 \text{ and } w_1^{red}, w_2^{red}\in K_{n-1}\} \text{. }\]
Let $w\in K_n$. We claim that any conjugate of $w$ by one of the generators is again in $K_n.$
For $a^{-1}wa$, this follows from the fact that $|a^{-1}wa|_a=|w|_a$ and that $(a^{-1}wa)_1=w_2$ and $(a^{-1}wa)_2=w_1$. For the remaining cases, we see that $|b^{-1}wb|_a=|c^{-1}wc|_a=|d^{-1}wd|_a=|w|_a$, and moreover,
\[
(b^{-1}wb)_1=a^{-1}w_1 a \qquad (b^{-1}wb)_2=c^{-1}w_2 c
\]
\[
(c^{-1}wc)_1=aw_1 a^{-1} \qquad (c^{-1}wc)_2 =d^{-1}w_2 d
\]
\[
(d^{-1}wd)_1=w_1 \qquad (d^{-1}wd)_2=b^{-1}w_2 b \text{.}
\]
Now normality follows from the induction hypothesis.

Finally, we verify that the inclusions are proper. Let $\sigma$ be the endomorphism of $F$ defined by $\sigma(a)=a^{-1}c^{-1}a, \sigma(b)=d, \sigma(c)=b,$ and $\sigma(d)=c$. Observe that $\sigma$ takes reduced words to reduced words and moreover if $|w|_a=0$ then $|\sigma(w)|_a=0$. Also note that $\sigma(a)_1=d^{-1}$, $\sigma(d)_1=a^{-1}$, $\sigma(a)_2=a$ and $\sigma(d)_2=d$. Now fix
\[
w=[a,d][a^{-1},d^{-1}]=a^{-1}d^{-1}adada^{-1}d^{-1}
\]
and
\[
\widetilde{w}=[d^{-1},a^{-1}][d,a]=dad^{-1}a^{-1}d^{-1}a^{-1}da \text{.}
\]
We claim that $\sigma^n(w)$ and $\sigma^n(\widetilde{w})$ are in $K_{n+1}$ but not $K_n$. For the base cases, one easily checks that $w$ and $\widetilde{w}$ are both in $K_1$ but not in $K_0$. Now note that $\sigma(w)_1=\widetilde{w}$, $\sigma(w)_2=w$, $\sigma(\widetilde{w})_1=w$, and $\sigma(\widetilde{w})_2=\widetilde{w}$, so using the base case we get that $\sigma(w)$ and $\sigma(\widetilde{w})$ are both in $K_2$ but not $K_1$. Continuing in this way, the result follows by induction.
\end{proof}

\begin{proposition}\label{prop:brGrig_not_fp}
The braided Grigorchuk group $\brGrig$ is not finitely presented.
\end{proposition}

\begin{proof}
Suppose $\brGrig$ is finitely presented. As it is a quotient of $F$, it has a presentation of the form
\[
\langle a,b,c,d \mid bcd,b^{-1}c^{-1}bc, r_1, r_2, \dots r_k\rangle\text{,}
\]
giving $\brGrig\cong F/R$ where $R=\ker(\Psi)$ is the normal closure of $r_1, \dots, r_k$ in $F$. As each $r_i$ is contained in some $K_n$, this contradicts Lemma~\ref{lem:proper}.
\end{proof}

Anthony Genevois has pointed out to us that Proposition~\ref{prop:brGrig_not_fp} also follows from \cite[Theorem~1.6]{benli13}, which implies that no finitely presented amenable group has $\Grig$ as a quotient.

\begin{remark}
Our braided Grigorchuk group $\brGrig$ is similar to the torsion-free group of intermediate growth constructed by Grigorchuk in \cite[Section~5]{grigorchuk85}; see \cite{allcock21} for a similar construction (which in discussions with Daniel Allcock we have determined is isomorphic to the group in \cite{grigorchuk85}). The difference is that, phrased in our language, it seems likely that that group would use the braided wreath recursions $b=(a,c)$, $c=(a,d)$, and $d=(1,b)$. We chose $c=(a^{-1},d)$ so that we would get $bcd=1$, which makes $\brGrig$ feel more closely related to $\Grig$. We will leave it as a question for future investigation whether $\brGrig$ is isomorphic to the group from \cite{grigorchuk85,allcock21}, and so in particular whether $\brGrig$ has intermediate growth.
\end{remark}

\section{Braiding groups of almost-automorphisms of trees}\label{sec:aauts}

In this section we discuss groups of almost-automorphisms of trees, and introduce braided versions. We will use the framework of $d$-ary cloning systems from \cite{skipper21}, which generalize cloning systems from \cite{witzel18}. Very loosely, $d$-ary cloning systems provide a way to construct new Thompson-like groups, and have proven useful in a variety of recent work, for example on decision problems \cite{berns-zieve18}, orderability \cite{ishida18}, von Neumann algebras \cite{bashwinger}, and the so called Jones technology \cite{brothier21}. Let us recall the relevant background.

\subsection{Cloning systems}\label{ssec:cloning}

\begin{definition}[$d$-ary cloning system]
Let $d\ge 2$ be an integer and $(G_n)_{n\in\N}$ a family of groups. For each $n\in\N$ let $\rho_n\colon G_n\to S_n$ be a homomorphism to the symmetric group $S_n$, called a \emph{representation map}. For each $1\le k\le n$ let $\clone_k^n \colon G_n \to G_{n+d-1}$ be an injective function (not necessarily a homomorphism), called a \emph{$d$-ary cloning map}. We write $\rho_n$ to the left of its input and $\clone_k^n$ to the right of its input, for reasons of visual clarity. Now we call the triple
\[
((G_n)_{n\in\N},(\rho_n)_{n\in\N},(\clone_k^n)_{k\le n})
\]
a \emph{$d$-ary cloning system} if the following axioms hold:\\
\indent\textbf{(C1):} (Cloning a product) $(gh)\clone_k^n = (g)\clone_{\rho_n(h)k}^{n} (h)\clone_k^n$\\
\indent\textbf{(C2):} (Product of clonings) $\clone_\ell^n \circ \clone_k^{n+d-1} = \clone_k^n \circ \clone_{\ell+d-1}^{n+d-1}$\\
\indent\textbf{(C3):} (Compatibility) $\rho_{n+d-1}((g)\clone_k^n)(i) = (\rho_n(g))\symmclone_k^n(i)$ for all $i\ne k,k+1,\dots,k+d-1$.

Here we always have $1\le k<\ell\le n$ and $g,h\in G_n$, and $\symmclone_k^n$ denotes the standard $d$-ary cloning maps for the symmetric groups.
\end{definition}

The maps $\symmclone_k^n$ are explained in \cite[Example~2.2]{skipper21}, and we will review them in Example~\ref{ex:symm_clone} below.

\begin{remark}
For an illustration of a cloning map in the special case $G_n=\brW_n$ (defined in Subsection~\ref{ssec:braid_aaut}), we direct the reader to Figure~\ref{fig:braaut_clone}.
\end{remark}

Given a $d$-ary cloning system on a family of groups $(G_n)_{n\in\N}$ one gets a Thompson-like group, denoted $\Thomp_d(G_*)$, which can be viewed as a sort of ``Thompson limit'' of the $G_n$. Let us recall the construction of $\Thomp_d(G_*)$. First, a \emph{$d$-ary tree} is a finite rooted tree in which each non-leaf vertex has $d$ children, and a \emph{$d$-ary caret} is a $d$-ary tree with $d$ leaves. An element of $\Thomp_d(G_*)$ is represented by a triple $(T_-,g,T_+)$ where $T_\pm$ are $d$-ary trees with the same number of leaves, say $n$, and $g$ is an element of $G_n$. There is an equivalence relation on such triples, and the equivalence classes are the elements of $\Thomp_d(G_*)$. The equivalence relation is given by expansion and reduction: an \emph{expansion} of $(T_-,g,T_+)$ is a triple of the form $(T_-',(g)\clone_k^n,T_+')$ where $T_+'$ is $T_+$ with a $d$-ary caret added to the $k$th leaf and $T_-'$ is $T_-$ with a $d$-ary caret added to the $\rho_n(g)(k)$th leaf. A \emph{reduction} is the reverse of an expansion. Now declare that two triples are equivalent if we can get from one to the other via a finite sequence of expansions and reductions, and write $[T_-,g,T_+]$ for the equivalence class of $(T_-,g,T_+)$.

\begin{definition}[Thompson-like group]
The \emph{Thompson-like group} $\Thomp_d(G_*)$ is the set of equivalence classes $[T_-,g,T_+]$.
\end{definition}

We have not explained the group operation on $\Thomp_d(G_*)$. The idea is that given any two elements $[T_-,g,T_+]$ and $[U_-,h,U_+]$, up to expansions we can assume $T_+=U_-$. This is because any pair of $d$-ary trees have a common $d$-ary tree obtainable from either of them by adding $d$-ary carets to their leaves. Now the group operation on $\Thomp_d(G_*)$ is defined by
\[
[T_-,g,T_+][U_-,h,U_+] \defeq [T_-,gh,U_+]
\]
when $T_+=U_-$. The cloning axioms ensure that this is a well defined group operation. The identity is $[T,1,T]$ (for any $T$) and inverses are given by $[T_-,g,T_+]^{-1} = [T_+,g^{-1},T_-]$.

\begin{example}[Cloning permutations]\label{ex:symm_clone}
The most fundamental example of a $d$-ary cloning system is on the family $(S_n)_{n\in\N}$ of symmetric groups. Take $\rho_n \colon S_n\to S_n$ to be the identity, and let
\[
\symmclone_k^n \colon S_n \to S_{n+d-1}
\]
be the function described as follows. Visualize $\sigma\in S_n$ by drawing arrows going up, from a line of labels $1$ through $n$ for the domain to another line of labels $1$ through $n$ for the range, with an arrow from $i$ to $\sigma(i)$ for each $i$. Now $(\sigma)\symmclone_k^n \in S_{n+d-1}$ is obtained by replacing the arrow from $k$ to $\sigma(k)$ by $d$ parallel arrows, and relabeling everything appropriately. For a (complicated) rigorous formula, see \cite[Example~2.2]{skipper21}. It turns out that this defines a $d$-ary cloning system, with $d$-ary cloning maps $\symmclone_k^n$. The Thompson-like group $\Thomp_d(S_*)$ that arises from this $d$-ary cloning system is (isomorphic to) the Higman--Thompson group $V_d$.
\end{example}

\begin{example}[Cloning braids]\label{ex:braid_clone}
The braided version of the above example is a $d$-ary cloning system on the family $(B_n)_{n\in\N}$ of braid groups. Take $\rho_n\colon B_n\to S_n$ to be the natural projection of $B_n$ onto $S_n$, and let
\[
\braidclone_k^n \colon B_n \to B_{n+d-1}
\]
be the function described as follows. Visualize $\beta\in B_n$ as an $n$-strand braid diagram, counting the strands $1$ through $n$ at the bottom. Now $(\beta)\braidclone_k^n \in B_{n+d-1}$ is obtained by replacing the $k$th strand (counting at the bottom) by $d$ parallel strands. As discussed in \cite[Remark~2.10]{witzel18}, in the $d=2$ case this really defines a cloning system (this was essentially already shown by Brin in \cite{brin07}, before cloning systems had been introduced, using the language of Zappa-Sz\'ep products), and it is easy to see that in the arbitrary $d$ case it defines a $d$-ary cloning system. The Thompson-like group $\Thomp_d(B_*)$ that arises from this $d$-ary cloning system is (isomorphic to) the braided Higman--Thompson group $\brV_d$, considered previously by Aroca and Cumplido in \cite{aroca22}, and the first author and Wu in \cite{skipperwu}.
\end{example}

\subsection{Almost-automorphisms}\label{ssec:aauts_def}

Every automorphism of $\tree_d$ induces a self-homeomorphism of the boundary $\partial\tree_d$, which is a $d$-ary Cantor space. The idea behind almost-automorphisms of $\tree_d$ is to consider self-homeomorphisms of $\partial\tree_d$ that locally ``look like'' they came from $\Aut(\tree_d)$. Far more detail and rigor can be found for example in \cite{leboudec17} and \cite[Subsection~1.3]{skipper21}. For our purposes here, we will use the isomorphism in \cite[Theorem~2.6]{skipper21} to simply define the group of almost-automorphisms of $\tree_d$ as the Thompson-like group arising from the following cloning system.

For $n\in\N$, let $W_n\defeq S_n \wr \Aut(\tree_d)$ (in \cite{skipper21} this was denoted $A_n$, but here we will use $W_n$). Here the wreath product has no subscript, and so this should be interpreted as meaning there are $n$ copies of $\Aut(\tree_d)$ and $S_n$ acts on $\{1,\dots,n\}$ in the standard way. Following \cite[Subsection~2.2]{skipper21}, we will define a $d$-ary cloning system on $(W_n)_{n\in\N}$. Let $\rho_n \colon W_n \to S_n$ be the natural epimorphism onto the $S_n$ term, i.e.,
\[
\rho_n(\sigma(f_1,\dots,f_n))\defeq \sigma \text{.}
\]
To define the $d$-ary cloning maps $\clone_k^n \colon W_n \to W_{n+d-1}$ we need some notation. Given $\sigma(f_1,\dots,f_n)\in W_n$ and $1\le k\le n$, write the wreath recursion of $f_k$ as $f_k=\rho(f_k)(f_k^1,\dots,f_k^d)$. Also write $\rho^{(k)}(f_k)\in S_{n+d-1}$ for the image of $\rho(f_k)\in S_d$ under the ($k$-dependent) monomorphism $S_d\to S_{n+d-1}$ induced by the inclusion $\{1,\dots,d\}\to\{1,\dots,n\}$ sending $j$ to $k+j-1$. Now we can define $\clone_k^n$ as follows:
\[
(\sigma(f_1,\dots,f_n))\clone_k^n \defeq (\sigma)\symmclone_k^n \rho^{(k)}(f_k) (f_1,\dots,f_{k-1},f_k^1,\dots,f_k^d,f_{k+1},\dots,f_n) \text{.}
\]

As proved in \cite[Proposition~2.4]{skipper21}, these $\rho_n$ and $\clone_k^n$ define a $d$-ary cloning system on the $W_n$.

\begin{definition}[Group of almost-automorphisms]
The \emph{group of almost-automorphisms} $\AAut(\tree_d)$ of $\tree_d$ is
\[
\AAut(\tree_d) \defeq \Thomp_d(W_*) \text{,}
\]
where $(W_n)_{n\in\N}$ is equipped with the above $d$-ary cloning system.
\end{definition}

Now for self-similar $G\le \Aut(\tree_d)$, consider the family $(S_n \wr G)_{n\in\N}$. Thanks to self-similarity, the restriction of $\clone_k^n$ to $S_n \wr G$ lands in $S_{n+d-1} \wr G$, which means the $d$-ary cloning system on the $W_n=S_n \wr \Aut(\tree_d)$ restricts to a $d$-ary cloning system on $S_n \wr G$.

\begin{definition}[R\"over--Nekrashevych group]\label{def:nekr}
For a self-similar group $G\le \Aut(\tree_d)$, the \emph{R\"over--Nekrashevych group} for $G$ is
\[
V_d(G) \defeq \Thomp_d(S_*\wr G) \text{,}
\]
where $(S_n\wr G)_{n\in\N}$ is equipped with the above $d$-ary cloning system.
\end{definition}

This definition agrees with the usual definition first given by Nekrashevych \cite{nekrashevych04}, thanks to \cite[Corollary~2.7]{skipper21}.

\begin{definition}[R\"over group]
The \emph{R\"over group} is the group $V_2(\Grig)$.
\end{definition}

The R\"over group $V_2(\Grig)$ was first constructed by R\"over in \cite{roever99}, and generalized by Nekrashevych in \cite{nekrashevych04} to the full family of R\"over--Nekrashevych groups $V_d(G)$. R\"over proved that $V_2(\Grig)$ is isomorphic to the abstract commensurator of $\Grig$, and is a finitely presented simple group \cite{roever99,roever02}. Belk and Matucci \cite{belk16} proved that it is even of type $\F_\infty$. Analogous results for certain $V_d(G)$ were proved by Nekrashevych \cite{nekrashevych04}, Farley and Hughes \cite{farley15}, and the authors \cite{skipper21}.

\subsection{Braiding groups of almost-automorphisms}\label{ssec:braid_aaut}

Now we will introduce braided R\"over--Nekrashevych groups. Starting with self-identical groups, this was previously done by Aroca and Cumplido in \cite{aroca22}, but for self-similar groups that are not self-identical, to the best of our knowledge this has not been done. In particular using the braided Grigorchuk group to construct a braided R\"over group is new. As a remark, Aroca and Cumplido's constructions could have been phrased in the language of cloning systems, as they mention in their introduction (though they did not use this framework outside their introduction), so one can view our approach here as a direct generalization of theirs.

\textbf{Convention:} The notation $B_n \wr G$, i.e., with no subscript on the $\wr$, will always mean $B_n \wr_{\{1,\dots,n\}} G$ defined via the action of the braid group $B_n$ on $\{1,\dots,n\}$ coming from the natural projection $B_n \to S_n$.

Recall that $\brAut(\tree_d)=B_d \wr_X^\infty B_d$ and  let $\brW_n\defeq B_n \wr \brAut(\tree_d)$. We want to define a $d$-ary cloning system on $(\brW_n)_{n\in\N}$. We will use the notation $\rho_n$ and $\clone_k^n$ like we did in the non-braided case for the cloning system on $(W_n)_{n\in\N}$, and no confusion should arise. Let $\rho_n\colon \brW_n \to S_n$ be the composition of $\brW_n \to W_n$ with $W_n \to S_n$, where the first map is induced by the standard epimorphism $B_n\to S_n$ together with $\pi\colon \brAut(\tree_d)\to \Aut(\tree_d)$, and the second map is the $\rho_n$ epimorphism from the non-braided case. To define the $d$-ary cloning maps $\clone_k^n \colon \brW_n \to \brW_{n+d-1}$ we need some notation. Given $\beta(f_1,\dots,f_n)\in \brW_n$ and $1\le k\le n$, write the braided wreath recursion of $f_k$ as $f_k=\phi(f_k)(f_k^1,\dots,f_k^d)$. Also write $\phi^{(k)}(f_k)\in B_{n+d-1}$ for the image of $\phi(f_k)\in B_d$ under the ($k$-dependent) monomorphism $B_d\to B_{n+d-1}$ induced by adding $k-1$ new unbraided strands on the left and $n-k$ new unbraided strands on the right. Now we can define $\clone_k^n$ as follows:
\[
(\beta(f_1,\dots,f_n))\clone_k^n \defeq (\beta)\braidclone_k^n \phi^{(k)}(f_k) (f_1,\dots,f_{k-1},f_k^1,\dots,f_k^d,f_{k+1},\dots,f_n) \text{.}
\]
Note that $\clone_k^n$ is injective, since $\braidclone_k^n$ is injective and $f_k$ is uniquely determined by its braided wreath recursion. See Figure~\ref{fig:braaut_clone} for an example.

\begin{figure}[htb]
 \centering
 \begin{tikzpicture}[line width=0.8pt]
  \draw (0,-2) to [out=90, in=-90] (1,0);
  \draw[white, line width=6pt] (1,-2) to [out=90, in=-90] (0,0);
  \draw[line width=1.5pt] (1,-2) to [out=90, in=-90] (0,0);
  \node at (0,-2.25) {$\id$}; \node at (1,-2.25) {$f$};
  \node at (2,-1) {$\stackrel{\clone_2^2}{\longrightarrow}$};
  
  \begin{scope}[xshift=3cm,yshift=1cm]
   \draw (0,-4) -- (0,-2);
   \draw (0,-2) to [out=90, in=-90] (2,0);
   \draw (1,-2)[white, line width=6pt] to [out=90, in=-90] (0,0);
   \draw (1,-2)[line width=1.5pt] to [out=90, in=-90] (0,0);
   \draw (2,-2)[white, line width=6pt] to [out=90, in=-90] (1,0);
   \draw (2,-2)[line width=1.5pt] to [out=90, in=-90] (1,0);
   \draw[line width=1.5pt] (1,-4) to [out=90, in=-90] (2,-2);
   \draw[white, line width=6pt] (2,-4) to [out=90, in=-90] (1,-2);
   \draw[line width=1.5pt] (2,-4) to [out=90, in=-90] (1,-2);
   \node at (0,-4.25) {$\id$}; \node at (1,-4.25) {$f$}; \node at (2,-4.25) {$f$};
  \end{scope}
 \end{tikzpicture}
 \caption{An example of $2$-ary cloning on $\brW_2$. Here $f\in\brAut(\tree_2)$ satisfies the braided wreath recursion $f=\zeta(f,f)$. The picture shows that $(\zeta(\id,f))\clone_2^2 = (\zeta)\braidclone_2^2\phi^{(2)}(f)(\id,f,f)$. We use thick lines to indicate the strand getting cloned and the resulting strands.}
 \label{fig:braaut_clone}
\end{figure}
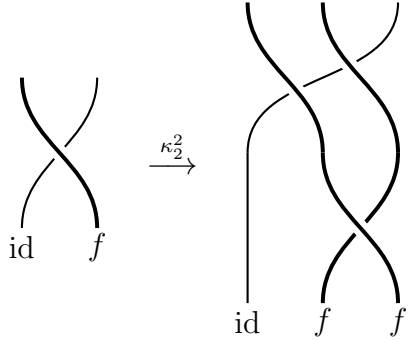

\begin{proposition}\label{prop:brW_n_cloning_system}
$((\brW_n)_{n\in\N},(\rho_n)_{n\in\N},(\clone_k^n)_{1\le k\le n})$ is a $d$-ary cloning system.
\end{proposition}

\begin{proof}
The proof is similar to that of \cite[Proposition~2.4]{skipper21} about the non-braided situation, and some parts work in exactly the same way. Hence, in the course of this proof, we will sometimes just state that a certain step works analogously to the corresponding step in \cite[Proposition~2.4]{skipper21}. First we prove (C1) (cloning a product). Let $f=\beta(f_1,\dots,f_n)$ and $g=\gamma(g_1,\dots,g_n)$ be elements of $\brW_n$, so the product $fg$ equals $\beta\gamma(f_{\gamma(1)}g_1,\dots,f_{\gamma(n)}g_n)$. Here $B_n$ acts on $\{1,\dots,n\}$ via the projection $B_n\to S_n$, and the notation $\gamma(i)$ indicates this action. Similar to the proof of \cite[Proposition~2.4]{skipper21}, the left-hand side of (C1) is
\begin{align*}
(fg)\clone_k^n &= (\beta\gamma)\braidclone_k^n \phi^{(k)}(f_{\gamma(k)}g_k)\\
&\times (f_{\gamma(1)}g_1,\dots,f_{\gamma(k-1)}g_{k-1},f_{\gamma(k)}^{\phi(g_k)(1)}g_k^1,\dots,f_{\gamma(k)}^{\phi(g_k)(d)}g_k^d,f_{\gamma(k+1)}g_{k+1},\dots,f_{\gamma(n)}g_n) \text{.}
\end{align*}
Here the braided wreath recursions of $f_{\gamma(k)}$ and $g_k$ are $f_{\gamma(k)}=\phi(f_{\gamma(k)})(f_{\gamma(k)}^1,\dots,f_{\gamma(k)}^d)$ and $g_k=\phi(g_k)(g_k^1,\dots,g_k^d)$, so the braided wreath recursion of $f_{\gamma(k)}g_k$ is
\[
f_{\gamma(k)}g_k = \phi(f_{\gamma(k)}g_k)(f_{\gamma(k)}^{\phi(g_k)(1)}g_k^1,\dots,f_{\gamma(k)}^{\phi(g_k)(d)}g_k^d) \text{.}
\]
Now we need to show that the right-hand side of (C1), which is $(f)\clone_{\gamma(k)}^n(g)\clone_k^n$, equals the same thing. One can compute (similar to the proof of \cite[Proposition~2.4]{skipper21}) that this equals
\begin{align*}
&(\beta)\braidclone_{\gamma(k)}^n \phi^{(\gamma(k))}(f_{\gamma(k)}) (\gamma)\braidclone_k^n \phi^{(k)}(g_k) \\
&\times (f_1,\dots,f_{\gamma(k)-1},f_{\gamma(k)}^1,\dots,f_{\gamma(k)}^d,f_{\gamma(k)+1},\dots,f_n)^{(\gamma)\braidclone_k^n \phi^{(k)}(g_k)} \\
&\times (g_1,\dots,g_{k-1},g_k^1,\dots,g_k^d,g_{k+1},\dots,g_n) \text{,}
\end{align*}
where the superscript indicates conjugation in $\brW_{n+d-1}$. By the same argument as in the proof of \cite[Proposition~2.4]{skipper21}, the ``tuple parts'' of the left- and right-hand sides of (C1) are the same, so we only need to show that the ``braid parts'' are the same, i.e., that $(\beta\gamma)\braidclone_k^n \phi^{(k)}(f_{\gamma(k)}g_k) = (\beta)\braidclone_{\gamma(k)}^n \phi^{(\gamma(k))}(f_{\gamma(k)}) (\gamma)\braidclone_k^n \phi^{(k)}(g_k)$. Since we already know the $\vartheta_k^n$ define a $d$-ary cloning system on braid groups, we know $(\beta\gamma)\braidclone_k^n=(\beta)\vartheta_{\gamma(k)}^n(\gamma)\vartheta_k^n$, and since $\phi$ is a homomorphism this means it suffices to show that $(\gamma)\braidclone_k^n \phi^{(k)}(f_{\gamma(k)}) = \phi^{(\gamma(k))}(f_{\gamma(k)}) (\gamma)\braidclone_k^n$. To see this, note that the $k$th through $(k+d-1)$st strands of $(\gamma)\vartheta_k^n$ are all parallel to each other, and these are the only strands of $\phi^{(k)}(f_{\gamma(k)})$ that can braid non-trivially; see Figure~\ref{fig:C1_for_brW}.

Next, (C2) follows by an exactly analogous argument to the proof of \cite[Proposition~2.4]{skipper21}.

Finally we turn to (C3). Let $\beta(f_1,\dots,f_n)\in \brW_n$ and $i\ne k,k+1,\dots,k+d-1$. We need to show that $\rho_{n+d-1}((\beta(f_1,\dots,f_n))\clone_k^n)(i) = (\rho_n(\beta(f_1,\dots,f_n)))\symmclone_k^n(i)$. Setting $\sigma=\pi(\beta)$, the right-hand side just equals $(\sigma)\symmclone_k^n(i)$. The left-hand side equals
\[
\rho_{n+d-1}((\beta)\braidclone_k^n\phi^{(k)}(f_k)(f_1,\dots,f_{k-1},f_k^1,\dots,f_k^d,f_{k+1},\dots,f_n))(i) \text{,}
\]
which is $\pi((\beta)\braidclone_k^n\phi^{(k)}(f_k))(i)$. Since $\pi(\phi^{(k)}(f_k))(i)=i$ (by virtue of $i\ne k,k+1,\dots,k+d-1$), this equals $\pi((\beta)\braidclone_k^n)(i)$. Since $\pi$ plays the role of $\rho_n$ in the $d$-ary cloning system on the braid groups using the $d$-ary cloning maps $\vartheta_k^n$, by (C3) for that $d$-ary cloning system we know $\pi((\beta)\braidclone_k^n)(i) = (\sigma)\symmclone_k^n(i)$ as desired.
\end{proof}

\begin{figure}[htb]
\centering
\begin{tikzpicture}[line width=0.8pt]
  \draw
   (0,0) -- (0,-1)   (1,0) -- (1,-1)   (2,0) -- (2,-1)   (3,0) -- (3,-1)   (4,0) -- (4,-1)   (5,0) -- (5,-1)   (6,0) -- (6,-1);
  \draw[white, line width=4pt]
   (2,-1) to [out=-90, in=90] (3,-3)   (3,-1) to [out=-90, in=90] (4,-3)   (4,-1) to [out=-90, in=90] (5,-3);
  \draw[dashed]
   (2,-1) to [out=-90, in=90] (3,-3)   (3,-1) to [out=-90, in=90] (4,-3)   (4,-1) to [out=-90, in=90] (5,-3);
  \draw
   (-0.2,-1) -- (6.2,-1) -- (6.2,-3) -- (-0.2,-3) -- (-0.2,-1);
  \draw
   (0,-3) -- (0,-4)   (1,-3) -- (1,-4)   (2,-3) -- (2,-4)   (3,-3) -- (3,-4)   (4,-3) -- (4,-4)   (5,-3) -- (5,-4)   (6,-3) -- (6,-4);
  \draw
   (0,-4) -- (0,-6)   (1,-4) -- (1,-6)   (2,-4) -- (2,-6)   (3,-5) -- (3,-6)   (4,-5) -- (4,-6)  (5,-5) -- (5,-6)   (6,-4) -- (6,-6);
  \draw
   (2.8,-4) -- (5.2,-4) -- (5.2,-5) -- (2.8,-5) -- (2.8,-4);
  \node at (1,-2) {$(\gamma)\vartheta_3^5$};
  \node at (4,-4.5) {$\phi(f_{\gamma(3)})$};
  \node at (7,-3) {$=$};
  
  \begin{scope}[xshift=8cm]
   \draw[white, line width=4pt]
    (2,-1) to [out=-90, in=90] (3,-3)   (3,-1) to [out=-90, in=90] (4,-3)   (4,-1) to [out=-90, in=90] (5,-3);
   \draw[dashed]
    (2,-3) to [out=-90, in=90] (3,-5)   (3,-3) to [out=-90, in=90] (4,-5)   (4,-3) to [out=-90, in=90] (5,-5);
   \draw
    (0,0) -- (0,-3)   (1,0) -- (1,-3)   (2,0) -- (2,-1)   (3,0) -- (3,-1)   (4,0) -- (4,-1)   (5,0) -- (5,-3)   (6,0) -- (6,-3);
   \draw
    (1.8,-1) -- (4.2,-1) -- (4.2,-2) -- (1.8,-2) -- (1.8,-1);
   \draw
    (2,-2) -- (2,-3)   (3,-2) -- (3,-3)   (4,-2) -- (4,-3);
   \draw
    (-0.2,-3) -- (6.2,-3) -- (6.2,-5) -- (-0.2,-5) -- (-0.2,-3);
   \draw
    (0,-5) -- (0,-6)   (1,-5) -- (1,-6)   (2,-5) -- (2,-6)   (3,-5) -- (3,-6)   (4,-5) -- (4,-6)  (5,-5) -- (5,-6)   (6,-5) -- (6,-6);
   \node at (1,-4) {$(\gamma)\vartheta_3^5$};
   \node at (3,-1.5) {$\phi(f_{\gamma(3)})$};
  \end{scope}
\end{tikzpicture}
\caption{An example of the last step of the verification of (C1) in the proof of Proposition~\ref{prop:brW_n_cloning_system}. Here $d=3$, $k=3$, and $n=5$.}
\label{fig:C1_for_brW}
\end{figure}
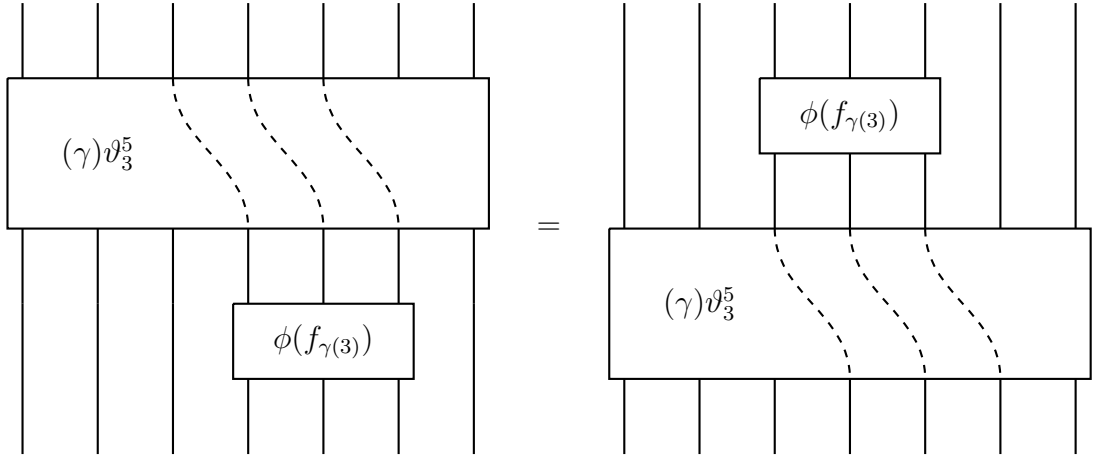

The Thompson-like group $\Thomp_d(\brW_*)$ is a braided version of $\Thomp_d(W_*)=\AAut(\tree_d)$, and we will denote it by $\brAAut(\tree_d)$.

Now for braided self-similar $G\le \brAut(\tree_d)$, consider the family $(B_n \wr G)_{n\in\N}$. Thanks to braided self-similarity, the restriction of $\clone_k^n$ to $B_n \wr G$ lands in $B_{n+d-1} \wr G$, which means the $d$-ary cloning system on the $\brW_n=B_n \wr \brAut(\tree_d)$ restricts to a $d$-ary cloning system on the $B_n \wr G$.

\begin{definition}[Braided R\"over--Nekrashevych group]\label{def:braided_ro_ne}
For a braided self-similar group $G\le \brAut(\tree_d)$, the \emph{braided R\"over--Nekrashevych group} for $G$ is
\[
\brV_d(G) \defeq \Thomp_d(B_*\wr G) \text{,}
\]
where $(B_n\wr G)_{n\in\N}$ is equipped with the above $d$-ary cloning system.
\end{definition}

Our main example is the following:

\begin{definition}[Braided R\"over group]
We define the \emph{braided R\"over group} to be the group $\brV_2(\brGrig)$.
\end{definition}

See Figure~\ref{fig:brR} for an example of an element of the braided R\"over group.

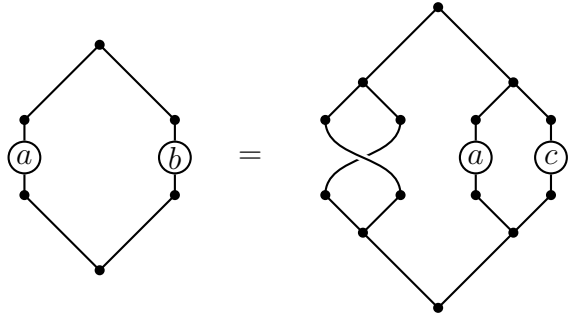
\begin{figure}[htb]
\centering
\begin{tikzpicture}[line width=0.8pt]
 \filldraw (0,0) circle (1.5pt)   (-1,-1) circle (1.5pt)   (1,-1) circle (1.5pt);
 \draw (-1,-1) -- (0,0) -- (1,-1);
 \filldraw (0,-3) circle (1.5pt)   (-1,-2) circle (1.5pt)   (1,-2) circle (1.5pt);
 \draw (-1,-2) -- (0,-3) -- (1,-2);
 
 \draw (-1,-1) -- (-1,-2)   (1,-1) -- (1,-2);
 
 \draw[white,fill=white] (-1,-1.5) circle (6 pt);
 \draw (-1,-1.5) circle (6 pt) node[text=black] {$a$};
 \draw[white,fill=white] (1,-1.5) circle (6 pt);
 \draw (1,-1.5) circle (6 pt) node[text=black] {$b$};
 
 \node at (2,-1.5) {$=$};
 
 \begin{scope}[xshift=4.5cm,yshift=0.5cm]
  \draw (-0.5,-1.5) to [out=-90, in=90] (-1.5,-2.5);
  \draw[line width=3pt, white] (-1.5,-1.5) to [out=-90, in=90] (-0.5,-2.5);
  \draw (-1.5,-1.5) to [out=-90, in=90] (-0.5,-2.5);
  
  \filldraw (0,0) circle (1.5pt)   (-1.5,-1.5) circle (1.5pt)   (1.5,-1.5) circle (1.5pt)   (-1,-1) circle (1.5pt)   (1,-1) circle (1.5pt)   (-0.5,-1.5) circle (1.5pt)   (0.5,-1.5) circle (1.5pt);
  \draw (-1.5,-1.5) -- (0,0) -- (1.5,-1.5)   (-0.5,-1.5) -- (-1,-1)   (0.5,-1.5) -- (1,-1);
  \filldraw (0,-4) circle (1.5pt)   (-1.5,-2.5) circle (1.5pt)   (1.5,-2.5) circle (1.5pt)   (-1,-3) circle (1.5pt)   (1,-3) circle (1.5pt)   (-0.5,-2.5) circle (1.5pt)   (0.5,-2.5) circle (1.5pt);;
  \draw (-1.5,-2.5) -- (0,-4) -- (1.5,-2.5)   (-0.5,-2.5) -- (-1,-3)   (0.5,-2.5) -- (1,-3);
  
  \draw (0.5,-1.5) -- (0.5,-2.5)   (1.5,-1.5) -- (1.5,-2.5);
  
  \draw[white,fill=white] (0.5,-2) circle (6 pt);
  \draw (0.5,-2) circle (6 pt) node[text=black] {$a$};
  \draw[white,fill=white] (1.5,-2) circle (6 pt);
  \draw (1.5,-2) circle (6 pt) node[text=black] {$c$};
 \end{scope}
\end{tikzpicture}
\caption{An element of the braided R\"over group. Here we draw a triple $(T_-,\beta(f_1,\dots,f_m),T_+)$ with $T_+$ upside down so that $\beta$ is a braid from the leaves of $T_-$ to the leaves of $T_+$, and the $f_i$ label the strands. The element equals $[\wedge,(a,b),\wedge]$, for $\wedge$ the tree with one caret. As indicated in the picture, after expansions this is the same element as $[T,\beta(1,1,a,c),T]$ for $T$ the result of adding a caret to each leaf of $\wedge$ and $\beta\in B_4$ the braid crossing the first strand over the second.}
\label{fig:brR}
\end{figure}

There is an obvious relationship between braided R\"over--Nekrashevych groups and (non-braided) R\"over--Nekrashevych groups, analogous to how the braided Thompson group $\brV$ surjects onto Thompson's group $V$. Given a braided R\"over--Nekrashevych group $\brV_d(G)$, we get a well defined epimorphism
\[
\pi\colon \brV_d(G)\to V_d(\pi(G)) \text{,}
\]
where $\pi(G)$ is the image of $G$ under $\pi\colon \brAut(\tree_d) \to \Aut(\tree_d)$. This epimorphism (which by abuse of notation we are also denoting by $\pi$) is given by sending $[T_-,\beta(f_1,\dots,f_n),T_+]$ to $[T_-,\pi(\beta)(\pi(f_1),\dots,\pi(f_n)),T_+]$. The $d$-ary cloning systems on $(B_n\wr G)_{n\in\N}$ and $(S_n\wr \pi(G))_{n\in\N}$ are clearly respected by $\pi$, so this map really is well defined.

\begin{remark}
The kernel of $\pi\colon \brV_d(G)\to V_d(\pi(G))$ consists of all $[T,\beta(f_1,\dots,f_n),T]$ such that $\beta$ is pure and each $f_i$ lies in the kernel of $\pi\colon \brAut(\tree_d)\to \Aut(\tree_d)$. By Lemma~\ref{lem:kernel_is_pure} each of these kernels is a direct product of copies of $PB_d$. Overall this means that the kernel of $\pi\colon \brV_d(G)\to V_d(\pi(G))$ is a direct limit of direct products of pure braid groups $PB_n$ with infinitely many copies of $PB_d$, with the direct limit informed by the cloning maps.
\end{remark}

\section{Finiteness properties}\label{sec:fin_props}

In this section we inspect finiteness properties of braided R\"over--Nekrashevych groups. The main result of this section is:

\begin{theorem}\label{thrm:braided_nekr_inherit_finprops}
Let $G\le \brAut(\tree_d)$ be braided self-similar. If $G$ is of type $\F_n$ then so is $\brV_d(G)$.
\end{theorem}

The case when $G$ is braided self-identical was already proved in \cite{skipperwu}, where it was shown that for braided self-identical $G$ the converse of Theorem~\ref{thrm:braided_nekr_inherit_finprops} is also true. (In general the converse of Theorem~\ref{thrm:braided_nekr_inherit_finprops} is not always true, as our Theorem~\ref{thrm:braided_roever_F_infty} will show.)

\subsection{Stein--Farley complexes}\label{ssec:stein_farley}

The starting point for deducing finiteness properties for a Thompson-like group is often to construct a so called Stein--Farley complex on which the group can act. For groups of the form $\Thomp_2(G_*)$, i.e., those arising from ($2$-ary) cloning systems, this was done in \cite{witzel18}. For groups of the form $\Thomp_d(S_*\wr G)$, i.e., R\"over--Nekrashevych groups, the Stein--Farley construction was done in \cite{skipper21}. For arbitrary groups of the form $\Thomp_d(G_*)$, i.e., those arising from $d$-ary cloning systems in general, the Stein--Farley construction has not technically been done in the literature, but it is easy to mimic the $2$-ary construction, and this is the goal of this subsection. Also see \cite{zaremsky21} for an introductory take on the situation when $d=2$ and $G_m=\{1\}$ for all $m$, i.e., for Thompson's group $F$. The name Stein--Farley complex pays homage to Stein \cite{stein92} and Farley \cite{farley03}, who constructed complexes like these for the classical Thompson groups and some close relatives.

Given a $d$-ary cloning system $((G_m)_{m\in\N},(\rho_m)_{m\in\N},(\clone_k^m)_{k\le m})$ we will construct a cube complex, denoted by $\Stein_d(G_*)$, called the \emph{Stein--Farley complex} for the $d$-ary cloning system. This is done in a number of steps.

\textbf{A groupoid:} First we consider the set of equivalence classes $[F_-,g,F_+]$, where $F_-$ is a $d$-ary forest, say with $m$ leaves (and some number of roots), $g$ is an element of $G_m$, and $F_+$ is a $d$-ary forest with $m$ leaves (and some number of roots). The equivalence relation on such triples is given by the expansion and reduction moves, as with elements of $\Thomp_d(G_*)$. This set is a groupoid; two elements $[F_-,g,F_+]$ and $[E_-,h,E_+]$ can be multiplied whenever the number of roots of $F_+$ equals the number of roots of $E_-$. In this case, up to expansions we can assume $F_+=E_-$, and the multiplication in this groupoid works analogously to the multiplication in $\Thomp_d(G_*)$. Multiplication is well defined in the groupoid for all the same reasons that it is well defined in $\Thomp_d(G_*)$. Denote this groupoid by $\Groupoid_d(G_*)$.

\textbf{A poset:} Now we restrict to equivalence classes of triples of the form $[T_-,g,F_+]$ for $T_-$ a $d$-ary tree, and mod out an additional equivalence relation given by right multiplication by elements of the form $[1,g',1]$, for $1$ a trivial forest with some number of roots, $m$, and $g'\in G_m$. Denote the resulting equivalence classes by $[T_-,g,F_+]_G$, and write $\Poset_d(G_*)$ for the set of all of them.  Given $[T_-,g,F_+]_G\in \Poset_d(G_*)$, say with $F_+$ having $m$ roots, and a groupoid element of the form $[F,1,1]$ for $F$ a $d$-ary forest with $m$ roots, it makes sense to consider the product $[T_-,g,F_+][F,1,1]$ and the equivalence class $[T_-,g,F_+][F,1,1]_G$. We define a partial order $\le$ on $\Poset_d(G_*)$ by declaring
\[
[T_-,g,F_+]_G\le [T_-,g,F_+][F,1,1]_G \text{.}
\]
(Note that thanks to the equivalence relation, in fact we have
\[
[T_-,g,F_+]_G\le [T_-,g,F_+][1,g',1][F,1,1]_G
\]
for any relevant $g'$.) It is easy to see that $\le$ is reflexive, transitive, and antisymmetric, hence really is a partial order. It is also clear that $\Poset_d(G_*)$ with $\le$ is a directed poset, since for any $[T_-,g,F_+]_G$ we have $[T_-,g,F_+]_G \le [T_-,g,1]_G=[T_-,1,1]_G$, and any two $[T_-,1,1]_G$ and $[T_-',1,1]_G$ have an upper bound. Hence the geometric realization $|\Poset_d(G_*)|$ is contractible.

\textbf{A cube complex:} Now we construct the Stein--Farley complex $\Stein_d(G_*)$, which is a cube complex having a natural subdivision identifiable with a certain subcomplex of $|\Poset_d(G_*)|$. The vertex set of $\Stein_d(G_*)$ is the whole vertex set of $|\Poset_d(G_*)|$, namely $\Poset_d(G_*)$. Let $E$ be a $d$-ary forest whose trees each have at most one $d$-caret; call such a $d$-ary forest \emph{elementary}. If $E$ is an elementary $d$-ary forest with one non-trivial tree and with the same number of roots as $F_+$, we can consider the product $[T_-,g,F_+][E,1,1]$ in the groupoid $\Groupoid_d(G_*)$. We declare that the vertices $[T_-,g,F_+]_G$ and $[T_-,g,F_+][E,1,1]_G$ span an edge in $\Stein_d(G_*)$. If $[T_-,g,F_+]_G\le [T_-,g,F_+][E,1,1]_G$ for $E$ elementary, write
\[
[T_-,g,F_+]_G\preceq [T_-,g,F_+][E,1,1]_G \text{.}
\]
More generally, if $E$ is any elementary $d$-ary forest with the same number of roots as $F_+$, then all the vertices of the form $[T_-,g,F_+][E',1,1]_G$, for $E'$ obtained from $E$ by replacing some number of $d$-carets with trivial trees, span the $1$-skeleton of a cube in $\Stein_d(G_*)$, and we declare that there is a cube in $\Stein_d(G_*)$ with this $1$-skeleton. All of this is well defined up to the equivalence relation. Any non-empty intersection of cubes is a common face of each of them, so $\Stein_d(G_*)$ really is a cube complex. Within a given cube, the geometric realization of the finite subposet of $\Poset_d(G_*)$ given by the vertices of the cube is naturally a subdivision of the cube. Hence there is a natural subdivision of $\Stein_d(G_*)$ that is identifiable with a subcomplex of $|\Poset_d(G_*)|$.

\textbf{A Morse function:} Let
\[
h\colon \Stein_d(G_*)^{(0)}\to\N
\]
be the function sending $[T_-,g,F_+]_G$ to the number of roots of $F_+$, and call this the \emph{number of feet} of this vertex. Note that adjacent vertices have distinct $h$ values. By the construction of $\Stein_d(G_*)$, it is clear that $h$ can be extended to a map $h\colon \Stein_d(G_*)\to\R$ that restricts to an affine map on each cube. Hence, $h$ satisfies all the requirements to be a discrete Morse function, in the sense of Bestvina--Brady \cite{bestvina97}. In particular we will be allowed to use discrete Morse theory when it comes up later. The function $h$ also comes into play when proving that $\Stein_d(G_*)$ is contractible, which we do shortly.

\textbf{Upward-local finiteness:} The main advantage of $\Stein_d(G_*)$ over $|\Poset_d(G_*)|$ is that the former is what we will call ``upward-locally finite'', as we now explain. Note that there are only finitely many elementary $d$-ary forests with a given number of roots. Also note that for any $[1,g,1],[F,1,1]\in \Groupoid_d(G_*)$ such that the product $[1,g,1][F,1,1]$ makes sense, there exist $F'$ and $g'$ such that $[1,g,1][F,1,1]=[F',1,1][1,g',1]$. More precisely, $F'$ and $g'$ are such that upon expanding we get $[1,g,1]=[F',g',F]$, and then
\[
[1,g,1][F,1,1]=[F',g',F][F,1,1]=[F',g',1]=[F',1,1][1,g',1]\text{.}
\]
In particular, given any vertex $x=[T_-,g,F_+]_G$ of $\Stein_d(G_*)$, every vertex $y$ with $x\preceq y$ is of the form $y=[T_-,g,F_+][E,1,1]_G$ for some elementary $d$-ary forest $E$. (Indeed, \emph{a priori} $y$ is of the form $[T_-,g,F_+][1,h,1][E,1,1]_G$, but the above shows that we can ignore the $[1,h,1]$ factor for the purposes of characterizing $y$.) We conclude that for any $x$ there exist only finitely many $y$ with $x\preceq y$. We refer to this property by saying that $\Stein_d(G_*)$ is \emph{upward-locally finite}.

\textbf{Contractibility:} To see that $\Stein_d(G_*)$ is contractible, we will show that the inclusion $\Stein_d(G_*)\to |\Poset_d(G_*)|$ is a homotopy equivalence. Using the usual notation of open and closed intervals in posets, every simplex of $|\Poset_d(G_*)|$ that is not in $\Stein_d(G_*)$ lies in the realization of a closed interval $[x,y]$ for $x=[T_-,g,F_+]_G\le y=[T_-,g,F_+][F,1,1]_G$ with $F$ a non-elementary $d$-ary forest. Hence it suffices to show that we can build up from $\Stein_d(G_*)$ to $|\Poset_d(G_*)|$ by gluing in the realizations of such closed intervals in some way that never changes the homotopy type. Using the above notation, the order we will use is in increasing order of the quantity $h(y)-h(x)$. Thus when we glue in $|[x,y]|$, we do so along $|[x,y)|\cup|(x,y]|$. This is the suspension of $|(x,y)|$, so it suffices to see that $|(x,y)|$ is contractible. For this we can mimic the proof of \cite[Lemma~4.7]{witzel18}, which is the $d=2$ case.  Recall that a poset $(Y, \sqsubseteq)$ is called \emph{conically contractible} if there is a $y_0$ in $Y$ and a poset map $g: Y \rightarrow Y$ such that $z \sqsubseteq g(z) \sqsupseteq y_0$ for all $z$ in $Y$. A consequence of a poset being conically contractible is that its geometric realization is contractible. Since a poset and its opposite poset have isomorphic geometric realizations, we can also use the criterion $z \sqsupseteq g(z) \sqsubseteq y_0$. See the discussion in \cite[Section 1.5]{quillen78} for more details. Now given any $z \in (x,y]$, define $g(z)$ to be the largest element of $[x,z]$ such that $x\preceq g(z)$ (the idea is to replace any $d$-ary forest with its unique largest elementary subforest). By our hypothesis, $g(z)$ is in  $[x,y)$ and it is also clearly in $(x,y]$ so we have $g(z)\in (x,y)$. Let $y_0=g(y).$ Note that for any $z\in (x,y)$, we have $g(z) \le y_0$. Therefore, $z \ge g(z) \le y_0$ and $(x,y)$ is conically contractible.

\textbf{Action and stabilizers:} The group $\Thomp_d(G_*)$ acts on $\Stein_d(G_*)$ by left multiplication, which is well defined since the equivalence relation defining vertices of $\Stein_d(G_*)$ and the partial order dictating which vertices span cubes are both given by right multiplication. Note too that $h$ is invariant under this action. We claim that the stabilizer of a vertex with $h$ value $m$ is isomorphic to $G_m$. Given such a vertex $[T_-,g,F_+]_G$ and an element $[U_-,h,U_+]$ of $\Thomp_d(G_*)$, we have $[U_-,h,U_+][T_-,g,F_+]_G=[T_-,g,F_+]_G$ if and only if $[F_+,g^{-1},T_-][U_-,h,U_+][T_-,g,F_+]=[1,h',1]$ for some $h'\in G_m$ (since $F_+$ has $m$ roots, the ``$1$'' here is the trivial forest with $m$ roots, and $h'\in G_m$). Thus $[U_-,h,U_+]\mapsto h'$ provides the desired isomorphism from the vertex stabilizer to $G_m$. Now we claim that any cube stabilizer is a finite index subgroup of a vertex stabilizer. Consider a cube in $\Stein_d(G_*)$, say with $x$ its unique vertex with minimum $h$ value and $y$ its unique vertex with maximal $h$ value. Since the action preserves $h$, the cube stabilizer lies in the stabilizer of the vertex $x$ (and that of $y$). Since $\Stein_d(G_*)$ is upward-locally finite, $\Symm(\{z\mid x\preceq z\})$ is finite, so the kernel of $\Stab(x)\to \Symm(\{z\mid x\preceq z\})$ has finite index in $\Stab(x)$. This kernel clearly fixes our cube pointwise, so we conclude that the stabilizer of the cube has finite index in the stabilizer of $x$.

\textbf{Cocompactness:} Let $\Stein_d(G_*)^{h\le m}$ be the full subcomplex of $\Stein_d(G_*)$ spanned by vertices with at most $m$ feet. Since $h$ is invariant under the action of $\Thomp_d(G_*)$, all the $\Stein_d(G_*)^{h\le m}$ are stabilized by this action. Given vertices $[T_-,g,F_+]_G$ and $[U_-,h,E_+]_G$ with the same $h$ value, the product $[U_-,h,E_+][F_+,g^{-1},T_-]$ exists, lies in $\Thomp_d(G_*)$, and takes $[T_-,g,F_+]_G$ to $[U_-,h,E_+]_G$. Hence $\Thomp_d(G_*)$ is transitive on vertices of a given $h$ value. Since $\Stein_d(G_*)$ is upward-locally finite, and since the vertices of any given $\Stein_d(G_*)^{h\le m}$ only have finitely many $h$ values, this shows that the action of $\Thomp_d(G_*)$ on any $\Stein_d(G_*)^{h\le m}$ is cocompact.

Let us summarize all of the above:

\begin{proposition}\label{prop:stein_farley}
The group $\Thomp_d(G_*)$ acts on the contractible cube complex $\Stein_d(G_*)$ with cube stabilizers isomorphic to finite index subgroups of the $G_m$. The action on each $\Stein_d(G_*)^{h\le m}$ is cocompact. \qed
\end{proposition}

\begin{remark}
In \cite{skipper21} there is a more general construction given for a self-similar group $G$, called the \emph{$H$-Stein--Farley complex} for $H\le G$, with the $H=G$ case recovering the original Stein--Farley complex. One could also fully generalize the $H$-Stein--Farley complexes from \cite{skipper21} as we have just done in the $H=G$ case, to some sort of $(H_*)$-Stein--Farley complexes for $H_m\le G_m$. This would be quite technical and tedious in general, and besides getting the braided R\"over group to be $\F_\infty$ it is unclear what this would buy us at the moment. Hence we will do this (in the following subsection) just for the braided R\"over group, where we can follow the comparatively easier roadmap from \cite{belk16}. 
\end{remark}

\subsection{Descending links}\label{ssec:dlks}

In this subsection and the next, we prove Theorem~\ref{thrm:braided_nekr_inherit_finprops}. In a now-standard approach, we will combine Proposition~\ref{prop:stein_farley} with Brown's Criterion and discrete Morse theory to reduce the problem to an analysis of descending links.

First let us recall some background on discrete Morse theory and descending links in general. Let $Y$ be an affine cell complex, in the sense of \cite{bestvina97}. Let $h\colon Y\to\R$ be a map such that the image of the vertex set of $Y$ is a closed, discrete subset of $\R$. If the restriction of $h$ to every positive-dimensional cell is a non-constant affine map, then we call $h$ a \emph{Morse function}. These conditions ensure that for every cell there is a unique vertex at which $h$ achieves its maximum value on that cell. The \emph{descending star} $\dst v$ of a vertex $v$ is the subcomplex of $Y$ consisting of all cells for which $v$ is this vertex with maximum $h$ value. The \emph{descending link} $\dlk v$ of $v$ is the link of $v$ in $\dst v$, i.e., the space of directions out of $v$ along which $h$ strictly decreases. (One could analogously define ascending stars and links, but here we will only use the descending version.)

The point of Morse theory is that an understanding of the descending links can translate to an understanding of the whole complex. The following essentially combines the Morse Lemma from \cite{bestvina97} with Brown's Criterion from \cite{brown87}, to produce a sufficient condition for a group to be of type $\F_n$. See \cite[Lemma~3.1]{skipper21} for an $\F_\infty$ version of the following, which works exactly analogously for $\F_n$.

\begin{cit}[Brown's Criterion plus discrete Morse theory]\label{cit:brown_morse}
Let $\Gamma$ be a group acting cellularly on an $(n-1)$-connected CW complex $X$. Suppose the stabilizer of any $p$-cell is of type $\F_{n-p}$. Let $h\colon X\to \R$ be a $\Gamma$-invariant Morse function such that the sublevel complexes $X^{h\le m}$ are $\Gamma$-cocompact for each $m\in\R$. If there exists $t\in\R$ such that for all $x\in X^{(0)}$ with $h(x)\ge t$ we have that $\dlk x$ is $(n-1)$-connected, then $\Gamma$ is of type $\F_n$.
\end{cit}

Now let us discuss descending links in $\Stein_d(G_*)$. Since $\Stein_d(G_*)$ is a cube complex, the link of a vertex $v$ is a simplicial complex, with a $k$-simplex for each $(k+1)$-cube containing $v$. The descending link is the subcomplex whose simplices correspond to cubes for which the vertex is the one with maximum $h$ value, i.e., the most feet. Thus, a $k$-simplex in $\dlk v$, say with $h(v)=m$, is represented by $v[1,g,E]_G$, for $g$ some element of $G_m$ and $E$ some elementary $d$-ary forest with $m$ leaves and $k+1$ carets. The face relation is given by removing carets from $E$. Note that the barycentric subdivision of $\dlk v$ is naturally a subcomplex of $|\Poset_d(G_*)|$, namely the full subcomplex spanned by all the $v[1,g,E]_G$.

There is a convenient model for the descending links. Note that, up to the action of $\Thomp_d(G_*)$, $\dlk v$ only depends on $h(v)$. Thus the following complex is isomorphic to every descending link of a vertex with $m$ feet:

\begin{definition}[Model for descending link]\label{def:dlk_model}
Let $\dlmodel_d^m(G_*)$ be the simplicial complex with a $k$-simplex for each $[1,g,E]_G$ for $g\in G_m$ and $E$ an elementary $d$-ary forest with $m$ leaves and $k+1$ carets, with face relation given by removing carets from $E$.
\end{definition}

The culmination of Subsections~\ref{ssec:stein_farley} and~\ref{ssec:dlks} is the following, which holds for any $d$-ary cloning system:

\begin{proposition}\label{prop:fin_props_machine}
Suppose $G_m$ is of type $\F_n$ for all $m$, and that $\dlmodel_d^m(G_*)$ is $(n-1)$-connected for all but finitely many $m$. Then $\Thomp_d(G_*)$ is of type $\F_n$.
\end{proposition}

\begin{proof}
We look at $\Thomp_d(G_*)$ acting on $\Stein_d(G_*)$, with the Morse function $h$, and verify the requirements from Citation~\ref{cit:brown_morse}. By Proposition~\ref{prop:stein_farley}, $\Stein_d(G_*)$ is contractible (hence $(n-1)$-connected), the stabilizer of any $p$-cube is isomorphic to a finite index subgroup of some $G_m$, so is of type $\F_n$ (hence $\F_{n-p}$) by assumption, and the sublevel complexes are all cocompact. Since $\dlmodel_d^m(G_*)$ is $(n-1)$-connected for all but finitely many $m$, the statement about descending links holds, and so we are done.
\end{proof}

\subsection{Descending links in the braided case}\label{ssec:braid_dlks}

Now we return to the special case from Subsection~\ref{ssec:braid_aaut}. We have a braided self-similar group $G\le \brAut(\tree_d)$ of type $\F_n$, and need to prove that $\brV_d(G)=\Thomp_d(B_*\wr G)$ is of type $\F_n$. As indicated by Proposition~\ref{prop:fin_props_machine}, the key thing to show is that $\dlmodel_d^m(B_*\wr G)$ is $(n-1)$-connected for all but finitely many $m$. To analyze the higher connectivity of $\dlmodel_d^m(B_*\wr G)$, we will use a strategy that essentially comes from \cite{bux16} in the case when $G=\{1\}$ and $d=2$, and more generally from \cite{skipperwu} in the case when $G$ is braided self-identical.

The idea is to map $\dlmodel_d^m(B_*\wr G)$ to a complex that is easier to understand. We will use the following complex, considered in \cite{skipperwu}.

\begin{definition}[$d$-marked point disk complex]
Let $\surface_{b,m}^g$ be a surface with $b$ boundary components, $m$ marked points, and genus $g$. The \emph{$d$-marked-point-disk complex} $\diskcpx_d(\surface_{b,m}^g)$ is the simplicial complex defined as follows. A $k$-simplex is an isotopy class of $k+1$ disjointly embedded disks such that each disk encloses precisely $d$ marked points in its interior and has no marked points in its boundary. (Here whenever we say ``isotopy'' it is implicit that all the intermediate maps in the isotopy must satisfy these rules as well, e.g., marked points cannot drift out of the interior of the disk and then back in.) The face relation is given by taking subsets.
\end{definition}

\begin{remark}
This complex is related to the curve complex first defined by Harvey in \cite{Hav81}, namely, outside some low-complexity cases, $\diskcpx_d(\surface_{b,m}^g)$ can be viewed as a full subcomplex of the curve complex. Indeed, given a disk enclosing $d$ marked points we can take its boundary curve, which ``usually'' gives a vertex in the curve complex. This fails if the boundary curve bounds a disk, a punctured sphere, or an annulus on the other side, which can only happen if $m\le d+1$. It might also happen that two disks are disjoint up to isotopy but their boundary curves are isotopic, which can happen when $m=2d$.
\end{remark}

Our next goal is to map $\dlmodel_d^m(B_*\wr G)$ to $\diskcpx_d(\surface_{1,m}^0)$ in a certain way. First we describe a useful alternate viewpoint of elementary $d$-ary forests. Let $L_{m-1}$ be the \emph{linear graph} with $m$ vertices, that is the graph with $m$ vertices labeled $1$ through $m$, and $m-1$ edges, one connecting $i$ to $i+1$ for each $1\le i<m$. Call a subgraph of $L_{m-1}$ a \emph{$d$-matching} on $L_{m-1}$ if each of its connected components is a subgraph of length $d-1$. The set of $d$-matchings forms a simplicial complex called the \emph{$d$-matching complex}, denoted by $\CM_d(L_{m-1})$, where a matching forms a $k$-simplex whenever it consists of $k+1$ disjoint paths, and the face relation is given by inclusion. Observe that there is a bijection between the set of elementary $d$-ary forests with $m$ leaves and the set of $d$-matchings on $L_{m-1}$. Under the identification, $d$-carets correspond to paths of length $d-1$. See Figure~\ref{fig:matchings_to_forestsrep} for an example.

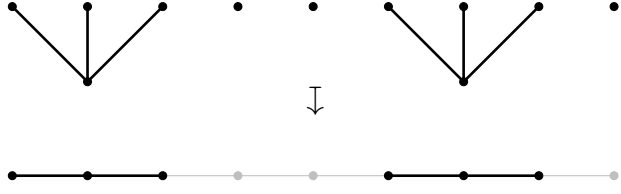
\begin{figure}[htb]
\centering
\begin{tikzpicture}
  \filldraw
   (0,0) circle (1.5pt)
   (1,0) circle (1.5pt)
   (2,0) circle (1.5pt)
   (3,0) circle (1.5pt)
   (4,0) circle (1.5pt)
   (5,0) circle (1.5pt)
   (6,0) circle (1.5pt)
   (7,0) circle (1.5pt)
   (8,0) circle (1.5pt)

   (1,-1) circle (1.5pt)
   (6,-1) circle (1.5pt);

  \draw[line width=1pt]
  (0,0) -- (1, -1)
  (1,0) -- (1, -1)
  (2,0) -- (1, -1)
  (5,0) -- (6, -1)
  (6,0) -- (6, -1)
  (7,0) -- (6, -1);
  
  \node[rotate=-90] at (4,-1.25) {$\mapsto$};

 \begin{scope}[yshift=-2.25cm]
 \draw[lightgray] (0,0) -- (8,0);
 
  \filldraw[lightgray]
   (3,0) circle (1.5pt)
   (4,0) circle (1.5pt)
   (8,0) circle (1.5pt);
  
  \filldraw
   (0,0) circle (1.5pt)
   (1,0) circle (1.5pt)
   (2,0) circle (1.5pt)
   (5,0) circle (1.5pt)
   (6,0) circle (1.5pt)
   (7,0) circle (1.5pt);
  
  \draw[line width=1pt]
   (0,0) -- (1,0)  -- (2,0) 
   (5,0) --  (6,0) -- (7,0);
 \end{scope}
\end{tikzpicture}
\caption{An example of the bijective correspondence between elementary $3$-ary forests with $9$ leaves and simplices of  $\CM_3(L_{8})$.}
\label{fig:matchings_to_forestsrep}
\end{figure}

Let $\surface_m=\surface_{1,m}^0$ be the unit disk with $m$ marked points given by fixing an embedding $L_{m-1} \hookrightarrow \surface_m$ of the linear graph with $m-1$ edges into $\surface_{1,m}^0$ (so the marked points are the images of the vertices). The braid group $B_m$ on $m$ strands is isomorphic to the mapping class group of the disk with $m$ marked points \cite{Bir74}, so we have an action of $B_m$ on $\diskcpx_{d}(\surface_m)$, which will be convenient to take to be a right action.

Define a map $\mu\colon \dlmodel_d^m(B_*\wr G) \rightarrow \diskcpx_{d}(\surface_m)$ as follows. For $[1, \beta (f_1, \dots, f_m), E]_G$ a simplex in $\dlmodel_d^m(B_*\wr G)$, first view the elementary $d$-ary forest $E$ as a $d$-matching in $\CM_d(L_{m-1})$ as above. Then take a disk tubular neighborhood of the $d$-matching, call it $\disk_E$, to obtain a simplex in $\diskcpx_{d}(\surface_m)$. Forget the labels $(f_1, \dots, f_m)$, and then act on $\disk_E$ by applying $\beta^{-1}$ on the right. In this way we obtain a simplicial map
\begin{align*}
\mu \colon  \dlmodel_d^m(B_*\wr G) & \to \diskcpx_{d}(\surface_m) \\
[1, \beta (f_1, \dots, f_m), E]_G & \mapsto (\disk_E)\beta^{-1} \text{ .}
\end{align*}

\begin{lemma}\label{lem:well_def}
The map $\mu$ is well defined.
\end{lemma}

\begin{proof}
Say $E$ has $q$ roots. Let $\gamma(g_1,\dots,g_q)\in B_q\wr G$, so
\[
[1, \beta (f_1, \dots, f_m), E]_G=[1,\beta(f_1, \dots, f_m), E][1,\gamma(g_1,\dots,g_q),1]_G \text{.}
\]
Let $\gamma'(g_1',\dots,g_m')\in B_q\wr G$ and $E'$ be such that
\[
[1,1,E][1,\gamma(g_1,\dots,g_q),1] = [1,\gamma'(g_1',\dots,g_m'),1][1,1,E'] \text{.}
\]
Then
\[
[1,\beta(f_1, \dots, f_m), E][1,\gamma(g_1,\dots,g_q),1]_G = [1,\beta\gamma'(f_1', \dots, f_m')(g_1',\dots,g_m'),1][1,1,E']_G\text{,}
\]
for some $f_i'$, so $\mu$ sends this to $(\disk_{E'})(\beta\gamma')^{-1}$. We need to show that this equals $(\disk_E)\beta^{-1}$, or equivalently that $(\disk_{E'})(\gamma')^{-1}=\disk_E$, i.e., $(\disk_E)\gamma'=\disk_{E'}$. Note that $\gamma'$ is obtained by taking $\gamma$, turning each strand corresponding to a root of $E$ with a $d$-caret on it into $d$ parallel strands (call this intermediate braid $\gamma''$), and then on each such collection of $d$ parallel strands applying some $\phi(g_i)\in B_d$ to it (where $\phi$ is as in Subsection~\ref{ssec:braid_aut}). When acting on $\disk_E$ though, since each of these local copies of the $\phi(g_i)$ are supported in the interior of one of the disks in the disk system, they do not change the (equivalence class of the) disk system. It is clear that $(\disk_E)\gamma''=\disk_{E'}$, so we conclude that also $(\disk_E)\gamma'=\disk_{E'}$.
\end{proof}

One can visualize $\mu$ as taking $[1,\beta(f_1,\dots,f_n),E]_G$, viewing it as the braid $\beta$ with strands labeled by the $f_i$ at the bottom and with $E$ serving to ``merge'' strands together, and then forgetting the labels $f_i$, considering the merges as $d$-matchings, enlarging them to disks, ``combing straight'' the braid, and seeing where the disks are taken. See Figure~\ref{fig:combing} for an example. Note that the resulting simplex $(\disk_E)\beta^{-1}$ of $\diskcpx_{d}(\surface_m)$ has the same dimension as the simplex $[1, \beta (f_1, \dots, f_m), E]_G$ of $\dlmodel_d^m(B_*\wr G)$, so in particular $\mu$ is simplexwise injective. Also note that if $d>2$ then the $\phi(g_i)$ from the proof of Lemma~\ref{lem:well_def} could change the $d$-matchings, but cannot change their corresponding disks, hence why we use disk complexes instead of, say, arc complexes.

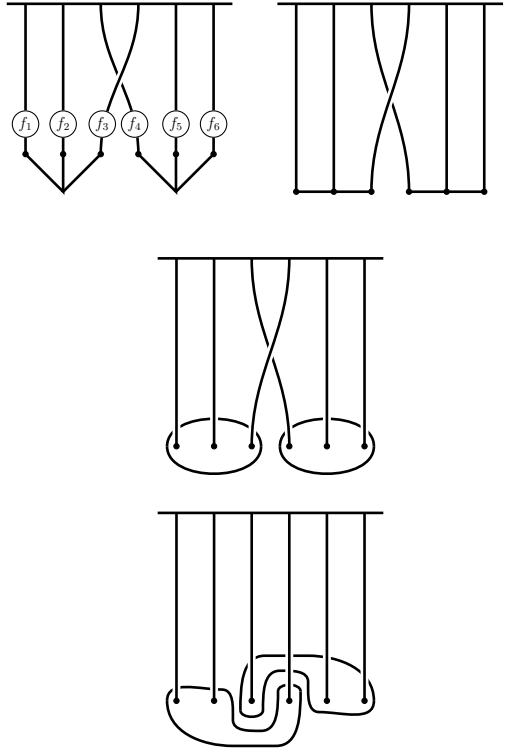
\begin{figure}[htb]
\centering
\begin{tikzpicture}
    \clip(-.30,0.4) rectangle (2.75,-2.95);

    \draw[line width=1pt] (0,0) to (0,-2);
    \draw[line width=1pt] (.5,0) to (.5,-2);
    
     \draw[line width=3pt, white] (1,0) to [out=270, in=90] (1.5,-2);
    \draw[line width=1pt] (1,0) to [out=270, in=90] (1.5,-2);
    
    \draw[line width=1pt] (2,0) to (2,-2);
    \draw[line width=1pt] (2.5,0) to (2.5,-2);
 
   \draw[line width= 1pt] (0,-2)-- (.5, -2.5)--(.5, -2);
   \draw[line width=1pt] (.5, -2.5) --(1,-2);
 
    \draw[line width= 1pt] (1.5,-2)-- (2, -2.5)--(2, -2);
   \draw[line width=1pt] (2, -2.5) --(2.5,-2);

        \draw[line width=3pt, white] (1.5,0) to [out=270, in=90] (1,-2);
    \draw[line width=1pt] (1.5,0) to [out=270, in=90] (1,-2);

         \draw[white,fill=white] (0,-1.6) circle (5 pt);
\draw  (0,-1.6) circle (5 pt) node[text=black, scale=.5] {$f_1$}; 

      \draw[white,fill=white] (0.5,-1.6) circle (5 pt);
\draw  (0.5,-1.6) circle (5 pt) node[text=black, scale=.5] {$f_2$}; 

       \draw[white,fill=white] (1.02,-1.6) circle (5 pt);
\draw  (1.02,-1.6) circle (5 pt) node[text=black, scale=.5] {$f_3$};  
   
        \draw[white,fill=white] (1.45,-1.6) circle (5 pt);
\draw  (1.45,-1.6) circle (5 pt) node[text=black, scale=.5] {$f_4$};  

       \draw[white,fill=white] (2,-1.6) circle (5 pt);
\draw  (2,-1.6) circle (5 pt) node[text=black, scale=.5] {$f_5$};  

       \draw[white,fill=white] (2.5,-1.6) circle (5 pt);
\draw  (2.5,-1.6) circle (5 pt) node[text=black, scale=.5] {$f_6$};

   \filldraw[]
   (0, -2) circle (1pt)
   (.5, -2) circle (1pt)
   (1, -2) circle (1pt)
   (1.5, -2) circle (1pt)
   (2, -2) circle (1pt)
   (2.5, -2) circle (1pt);

\draw[line width=1pt] (-0.25,0) -- (2.75,0);
    
\end{tikzpicture}
\quad
\begin{tikzpicture}
    \clip(-.30,0.4) rectangle (2.75,-2.95);

        \draw[line width=1pt] (0,0) to (0,-2.5);
    \draw[line width=1pt] (.5,0) to (.5,-2.5);

     \draw[line width=3pt, white] (1,0) to [out=270, in=90] (1.5,-2.5);
    \draw[line width=1pt] (1,0) to [out=270, in=90] (1.5,-2.5);
    
     \draw[line width=3pt, white] (1.5,0) to [out=270, in=90] (1,-2.5);
    \draw[line width=1pt] (1.5,0) to [out=270, in=90] (1,-2.5);
    
    \draw[line width=1pt] (2,0) to (2,-2.5);
    \draw[line width=1pt] (2.5,0) to (2.5,-2.5);
 
   \filldraw[]
   (0, -2.5) circle (1pt)
   (.5, -2.5) circle (1pt)
   (1, -2.5) circle (1pt)
   (1.5, -2.5) circle (1pt)
   (2, -2.5) circle (1pt)
   (2.5, -2.5) circle (1pt);
   
   \draw[line width=1pt, ] (0,-2.5) -- (1,-2.5);
   \draw[line width=1pt, ] (1.5,-2.5) -- (2.5,-2.5);

\draw[line width=1pt] (-0.25,0) -- (2.75,0);
\end{tikzpicture}

\quad
\begin{tikzpicture}
    \clip(-.30,0.4) rectangle (2.75,-2.95);
       \draw[line width=1pt] (-.125,-2.5) to [out=270, in=270] (1.125,-2.5);
       \draw[line width=1pt] (1.125,-2.5) to [out=90, in=90] (-.125,-2.5);

       \draw[line width=1pt] (1.375,-2.5) to [out=270, in=270] (2.625,-2.5);
       \draw[line width=1pt] (2.625,-2.5) to [out=90, in=90] (1.375,-2.5);
    
        \draw[line width=3pt, white] (0,0) to (0,-2.5);
        \draw[line width=1pt] (0,0) to (0,-2.5);
    \draw[line width=3pt, white] (.5,0) to (.5,-2.5);
    \draw[line width=1pt] (.5,0) to (.5,-2.5);

     \draw[line width=3pt, white] (1,0) to [out=270, in=90] (1.5,-2.5);
    \draw[line width=1pt] (1,0) to [out=270, in=90] (1.5,-2.5);
    
     \draw[line width=3pt, white] (1.5,0) to [out=270, in=90] (1,-2.5);
    \draw[line width=1pt] (1.5,0) to [out=270, in=90] (1,-2.5);
        
        \draw[line width=3pt, white] (2,0) to (2,-2.5);
    \draw[line width=1pt] (2,0) to (2,-2.5);
    
        \draw[line width=3pt, white] (2.5,0) to (2.5,-2.5);
    \draw[line width=1pt] (2.5,0) to (2.5,-2.5);
 
   \filldraw[]
   (0, -2.5) circle (1pt)
   (.5, -2.5) circle (1pt)
   (1, -2.5) circle (1pt)
   (1.5, -2.5) circle (1pt)
   (2, -2.5) circle (1pt)
   (2.5, -2.5) circle (1pt);
   
   \draw[line width=1pt] (-0.25,0) -- (2.75,0);

\end{tikzpicture}

\quad
\begin{tikzpicture}
 \clip(-.30,0.4) rectangle (2.75,-3.95);
 
\draw[line width=1pt] (-.125,-2.5) to [out=90, in=180] (.6,-2.35) to [out=0, in=90] (.75,-2.75)to [out=270, in=180] (1.125,-2.9) to [out=0, in=270] (1.35, -2.375) to [out=90, in=90] (1.65,-2.375);
\draw[line width=1pt] (-.125,-2.5) to [out=270, in=180] (.75,-3.1)to [out=0, in=180] (1.325,-3.1) to [out=0, in=270] (1.65,-2.375);

\begin{scope}[xscale=-1, xshift=-2.5cm, yscale=-1, yshift=5cm]
 \draw[line width=1pt] (-.125,-2.5) to [out=90, in=180] (.6,-2.35) to [out=0, in=90] (.75,-2.75)to [out=270, in=180] (1.125,-2.9) to [out=0, in=270] (1.35, -2.375) to [out=90, in=90] (1.65,-2.375);
\draw[line width=1pt] (-.125,-2.5) to [out=270, in=180] (.75,-3.1)to [out=0, in=180] (1.325,-3.1) to [out=0, in=270] (1.65,-2.375);
\end{scope}

    \draw[line width=3pt, white] (0,0) to (0,-2.5);
    \draw[line width=1pt] (0,0) to (0,-2.5);
    \draw[line width=3pt, white] (.5,0) to (.5,-2.5);
    \draw[line width=1pt] (.5,0) to (.5,-2.5);
    \draw[line width=3pt, white] (1,0) to (1,-2.5);
    \draw[line width=1pt] (1,0) to (1,-2.5);
    \draw[line width=3pt,white] (1.5,0) to (1.5,-2.5);
    \draw[line width=1pt] (1.5,0) to (1.5,-2.5);
    \draw[line width=3pt, white] (2,0) to (2,-2.5);
    \draw[line width=1pt] (2,0) to (2,-2.5);
    \draw[line width=3pt, white] (2.5,0) to (2.5,-2.5);
    \draw[line width=1pt] (2.5,0) to (2.5,-2.5);

  \filldraw[]
   (0, -2.5) circle (1pt)
   (.5, -2.5) circle (1pt)
   (1, -2.5) circle (1pt)
   (1.5, -2.5) circle (1pt)
   (2, -2.5) circle (1pt)
   (2.5, -2.5) circle (1pt);

\draw[line width=1pt] (-0.25,0) -- (2.75,0);

\end{tikzpicture}
\caption{An illustration of $\mu\colon \dlmodel_3^6(B_*\wr G) \rightarrow \diskcpx_3(\surface_6)$. The successive pictures show the process of flattening to a $3$-matching, expanding the matching to disks, forgetting the labels, and ``combing straight'' the braid.}
\label{fig:combing}
\end{figure}

In the following proofs, we will need a useful tool introduced by Hatcher and Wahl called the complete join \cite{HW10}.

\begin{definition}[Complete join (complex)]
A surjective simplicial map $\nu: Y\to X$ is called a \emph{complete join} if it satisfies the following properties:
\begin{enumerate} [label=(\arabic*)]
    \item $\nu$ is simplexwise injective.
    \item For each $k$-simplex $\sigma$ of $X$, say with vertices $v_0,\dots,v_k$, the fiber $\nu^{-1}(\sigma)$ equals the join $\nu^{-1}(v_0)\ast \cdots\ast \nu^{-1}(v_k)$.
 \end{enumerate}
In this case we call $Y$ a \emph{complete join complex over $X$}.
\end{definition}

\begin{definition}[Weakly Cohen-Macaulay (wCM)]
A simplicial complex  $X$ is called \emph{weakly Cohen-Macaulay (wCM) of dimension $n$} if $X$ is $(n-1)$-connected and the link of each $k$-simplex of $X$ is $(n-k-2)$-connected. (Note that $X$ need not necessarily be $n$-dimensional.)
\end{definition}

The main result regarding complete joins that we will use is the following.

\begin{cit}\cite[Proposition 3.5]{HW10}\label{cit:join}
If $Y$ is a complete join complex over a complex $X$ that is wCM of dimension $n$, then $Y$ is also wCM of dimension $n$.
\end{cit}

Now we would like to use the map $\mu\colon \dlmodel_d^m(B_*\wr G) \to \diskcpx_{d}(\surface_{1,m}^0)$ to prove higher connectivity of $\dlmodel_d^m(B_*\wr G)$, and hence of the descending links in the Stein--Farley complex. First we need the following, which is proved in \cite{skipperwu}:

\begin{cit}\cite{skipperwu}\label{cit:disk_conn}
For any $d\ge 2$, the complex $\diskcpx_{d}(\surface_{b,m}^g)$ is $(\left\lfloor\frac{m+1}{2d-1}\right\rfloor-2)$-connected.
\end{cit}

\begin{corollary}\label{cor:disk_conn}
For any $d\ge 2$, the complex $\diskcpx_{d}(\surface_{b,m}^g)$ is wCM of dimension $\lfloor\frac{m+1}{2d-1}\rfloor-1$.
\end{corollary}

\begin{proof}
Since $\diskcpx_{d}(\surface_{b,m}^g)$ is $(\left\lfloor\frac{m+1}{2d-1}\right\rfloor-2)$-connected by Citation~\ref{cit:disk_conn}, we just need to prove that the link of any $k$-simplex is $(\left\lfloor\frac{m+1}{2d-1}\right\rfloor-k-3)$-connected. Let $\sigma=\{D_0,\dots,D_k\}$ be a $k$-simplex in $\diskcpx_{d}(\surface_{b,m}^g)$. Let $S'$ be the surface obtained from $\surface_{b,m}^g$ by removing the interiors of the $D_i$, including the $d(k+1)$ many marked points in their interiors, leaving $k+1$ new boundary components. Since the $D_i$ are pairwise disjoint, $S'$ is connected, and hence is of the form $\surface_{b+(k+1),m-d(k+1)}^g$. The link of $\sigma$ is therefore isomorphic to $\diskcpx_{d}(\surface_{b+(k+1),m-d(k+1)}^g)$, which is $(\left\lfloor\frac{m-d(k+1)+1}{2d-1}\right\rfloor-2)$-connected by Citation~\ref{cit:disk_conn}. Since $\frac{-d(k+1)}{2d-1}\ge-(k+1)$, this is also $(\left\lfloor\frac{m+1}{2d-1}\right\rfloor-k-3)$-connected, so we are done.
\end{proof}

\begin{proposition}\label{prop:dlk_conn}
For any $d\ge 2$ and any $m\in\N$, the complex $\dlmodel_d^m(B_*\wr G)$ is wCM of dimension $\lfloor\frac{m+1}{2d-1}\rfloor-1$, so in particular is $(\lfloor\frac{m+1}{2d-1}\rfloor-2)$-connected.
\end{proposition}

\begin{proof}
Since $\diskcpx_{d}(\surface_{1,m}^0)$ is wCM of dimension $\lfloor\frac{m+1}{2d-1}\rfloor-1$ by Corollary~\ref{cor:disk_conn}, it suffices by Citation~\ref{cit:join} to prove that $\mu\colon \dlmodel_d^m(B_*\wr G) \to \diskcpx_{d}(\surface_{1,m}^0)$ is a complete join. We know that $\mu$ is surjective and simplexwise injective. Now we need to show that the fiber under $\mu$ of any simplex in $\diskcpx_{d}(\surface_{1,m}^0)$ is the join of the fibers of its vertices. Clearly the fiber of the simplex lies in the join of the vertex fibers, so we need to prove the reverse inclusion. This means that we need to prove that any collection of $k+1$ vertices in $\dlmodel_d^m(B_*\wr G)$ whose images under $\mu$ span a $k$-simplex in $\diskcpx_{d}(\surface_{1,m}^0)$ themselves span a $k$-simplex in $\dlmodel_d^m(B_*\wr G)$. We induct on $k$. The base case $k=0$ is trivial, so assume $k\ge 1$. Let $v_0,\dots,v_k$ be $k+1$ vertices in $\dlmodel_d^m(B_*\wr G)$ whose images under $\mu$ span a $k$-simplex in $\diskcpx_{d}(\surface_{1,m}^0)$. By induction, $v_1,\dots,v_k$ span a $(k-1)$-simplex, call it $\sigma$. Now $\mu(\sigma)$ is a collection of $k$ pairwise disjoint disks, and $\mu(v_0)$ is a disk disjoint from all the disks in $\mu(\sigma)$. Up to the left action of $B_m\wr G$ on $\dlmodel_d^m(B_*\wr G)$, we can assume without loss of generality that $\sigma$ is of the form $[1,1,E]_G$ for some $E$. Now say $v_0=[1,\beta(f_1,\dots,f_m),e]_G$ for some $\beta(f_1,\dots,f_m)\in B_m\wr G$ and $e$ some elementary $d$-ary forest with exactly one $d$-caret. Since $\mu(\sigma)=\disk_E$ is disjoint from $\mu(v_0)=(\disk_e)\beta^{-1}$, we know that $[1,\beta(f_1,\dots,f_m),E]_G=[1,1,E]_G$. Setting $E'$ equal to the elementary $d$-ary forest obtained as the union of $E$ and $e$, we therefore have that $[1,\beta(f_1,\dots,f_m),E']_G$ is a $k$-simplex containing both $\sigma$ and $v_0$.
\end{proof}

Now we can prove Theorem~\ref{thrm:braided_nekr_inherit_finprops}.

\begin{proof}[Proof of Theorem~\ref{thrm:braided_nekr_inherit_finprops}]
We have a braided self-similar group $G\le \brAut(\tree_d)$ of type $\F_n$, and need to prove that $\brV_d(G)=\Thomp_d(B_*\wr G)$ is of type $\F_n$. Note that $G$ being of type $\F_n$ implies each $B_m\wr G$ is of type $\F_n$, since a direct product of finitely many type $\F_n$ groups is of type $\F_n$, and an extension of a type $\F_n$ group (like $B_d$) by a type $\F_n$ group is type $\F_n$; see, e.g., \cite[Theorem~7.2.21]{geoghegan08}. Hence by Proposition~\ref{prop:fin_props_machine}, we just need to prove that $\dlmodel_d^m(B_*\wr G)$ is $(n-1)$-connected for all but finitely many $m$. This follows from Proposition~\ref{prop:dlk_conn}, since $\lfloor\frac{m+1}{2d-1}\rfloor-2$ goes to $\infty$ with $m$.
\end{proof}

\section{Finiteness properties of the braided R\"over group}\label{sec:braided_roever_fin_props}

This section is entirely about the braided R\"over group $\brV_2(\brGrig)$. For brevity let us write $\brR$ for the group, and we will also introduce other concise notation as we go. The main result of this section is the following:

\begin{theorem}\label{thrm:braided_roever_F_infty}
The braided R\"over group $\brR$ is of type $\F_\infty$.
\end{theorem}

Note that $\brGrig$ is not even finitely presented (Lemma~\ref{prop:brGrig_not_fp}) much less $\F_\infty$, so Theorem~\ref{thrm:braided_nekr_inherit_finprops} does not apply. In particular $\brR$ shows that the converse of Theorem~\ref{thrm:braided_nekr_inherit_finprops} is not true in general, in contrast to the braided self-identical case done in \cite{skipperwu}. Our roadmap is Belk and Matucci's proof that the R\"over group $V_2(\Grig)$ is of type $\F_\infty$ \cite{belk16}. We will also make use of the constructions from Section~\ref{sec:fin_props}.

To begin, we return to the groupoid $\Groupoid\defeq \Groupoid_2(B_*\wr \brGrig)$ of elements of the form $[F_-,g,F_+]$, for $F_\pm$ forests with $m$ leaves and $g\in B_m\wr \brGrig$. (In this section all trees and forests are binary.) As before, consider elements of the form $[T_-,g,F_+]$ (for $T_-$ a tree), and mod out an equivalence relation given by right multiplication by elements of the form $[1,g',1]$, but this time only do so for $g=\beta(f_1,\dots,f_m)$ with
\[
f_1,\dots,f_m\in Z\defeq \langle b,c,d\rangle \le \brGrig\text{.}
\]
The reason for doing this is that $Z\cong\Z^2$ is of type $\F_\infty$, unlike $\brGrig$ itself. Write
\[
[T_-,g,F_+]_Z
\]
for the resulting equivalence classes, and write $\sim_Z$ for the equivalence relation.

As before, define a partial order on these equivalence classes $[\tau]_Z$ (for $\tau$ a triple of the form $\tau=(T_-,g,F_+)$) by declaring that $[\tau]_Z \le [\tau][F,1,1]_Z$ for any forest $F$ with $m$ roots. Note that we must account for changing representatives up to $\sim_Z$, so we moreover have $[\tau]_Z \le [\tau][1,g',1][F,1,1]_Z$ for any $g'\in B_m\wr Z$. This is not transitive as defined, since we must account for changing representatives up to $\sim_Z$, so we actually take the transitive closure. Thus an upper bound of $[\tau]_Z$ looks like $[\tau][F_1,g_1',1]\cdots[F_k,g_k',1]_Z$ for $F_i$ some forests and $g_i'$ elements of the relevant braid groups wreathed with $Z$ (not with all of $\brGrig$, to reiterate). Write $\Poset_2(B_*\wr\brGrig)_Z$ for this poset, or $\Poset_Z$ for short.
\begin{lemma}\label{lem:contract}
For any $g\in B_q\wr\brGrig$, there exists a forest $F$ such that $[1,g,1][F,1,1]=[F',g',1]$ for some $F'$ and some $g'$ not only in $B_m\wr \brGrig$ but even in $B_m\wr Z$ for appropriate $m$.
\end{lemma}

\begin{proof}
Intuitively, this means that we can apply successive braided wreath recursions to $g$ until all the ``$a$''s are gone.

First note that if $g\in B_q$ then the result holds for any $F$, not just some $F$, so without loss of generality $g$ has no ``braid part'', i.e., $g\in \prod_q\brGrig$ (and at this point the desired $F'$ will be $F$). Now clearly without loss of generality $q=1$, so $g\in\brGrig$, and we are looking for a tree $T$ such that $[1,g,1][T,1,1]=[T,g',1]$ for $g'\in B_m\wr Z$. Write
\[
g=z_1 a^{k_1} z_2\cdots a^{k_\ell}z_{\ell+1}
\]
for $z_i\in Z$ and $k_i\in\Z$, as a reduced expression in the sense of Lemma~\ref{lem:br_grig_contract}. We will induct on the length of this expression. If the length is $1$ and $g\in Z$, we can just take $T$ to be trivial, and if the length is $1$ and $g$ is a power of $a$ then we can take $T$ to be a single caret. Now assume the length is greater than $1$. Let $\wedge$ be the tree with one caret, so $[1,g,1]=[\wedge,(g)\clone_1^1,\wedge]$. We have that $(g)\clone_1^1$ is of the form $\zeta^k(g_1,g_2)$ for some $k\in\Z$ and $g_1,g_2\in\brGrig$, and by Lemma~\ref{lem:br_grig_contract} the length of each of $g_1$ and $g_2$ is less than that of $g$. By induction we can therefore choose a tree $T$ such that $[1,g,1]=[T,g',T]$ with $g'\in B_m\wr Z$ for some $m$. Now $[1,g,1][T,1,1]=[T,g',1]$ as desired.
\end{proof}

\begin{corollary}
The poset $\Poset_Z$ is directed, so its geometric realization is contractible.
\end{corollary}

\begin{proof}
Let $[T_-,g,F_+]_Z$ and $[U_-,h,D_+]_Z$ be elements of $\Poset_Z$. Up to taking upper bounds, we can assume $F_+$ and $D_+$ are trivial. Now observe that thanks to Lemma~\ref{lem:contract} there exist forests $F$ and $F'$ such that $[1,g,1][F,1,1]=[F',g',1]$ for $g'$ not only in $B_m\wr \brGrig$ (for appropriate $m$) but even in $B_m\wr Z$. Hence $[1,g,1][F,1,1]_Z=[F',1,1]_Z$. This shows that, up to taking upper bounds, we can assume $g$ and $h$ are trivial. At this point our elements are $[T_-,1,1]_Z$ and $[U_-,1,1]_Z$, which clearly have an common upper bound.
\end{proof}

Our next goal is to construct an analog of the Stein--Farley complex. It will no longer have a cubical structure, but it will have a polysimplicial structure. First we need some definitions and notation.

\begin{definition}[Direct sum]
Given braids $\beta\in B_m$ and $\beta'\in B_{m'}$, the \emph{direct sum} $\beta\oplus\beta'$ is the element of $B_{m+m'}$ obtained by setting $\beta$ and $\beta'$ next to each other, so the first $m$ strands braid according to $\beta$, and the last $m'$ strands braid according to $\beta'$. Given a forest $F$ with $q$ roots and a forest $F'$ with $q'$ roots, the \emph{direct sum} $F\oplus F'$ is the forest with $q+q'$ roots whose first $q$ trees comprise $F$ and whose last $q'$ trees comprise $F'$. Now let $[F_-,g,F_+],[F_-',g',F_+']\in\Groupoid$. Say $g=\beta(f_1,\dots,f_m)$ and $g'=\beta'(f_1',\dots,f_{m'}')$. Let $F_-''=F_-\oplus F_-'$, let $F_+''=F_+\oplus F_+'$, let $\beta''=\beta\oplus\beta'$, and let $g''=\beta''(f_1,\dots,f_m,f_1',\dots,f_{m'}')$. Then the \emph{direct sum} $[F_-,g,F_+]\oplus [F_-',g',F_+']$ is the element $[F_-'',g'',F_+'']$ of $\Groupoid$. 
\end{definition}

\begin{definition}[Splitting, simple splitting]
Call an element of $\Groupoid$ of the form
\[
\splitting=[F_1,g_1,1]\cdots[F_k,g_k,1]
\]
for $g_i\in B_{m_i}\wr Z$ (for appropriate $m_i$) a \emph{splitting}, so in the poset $\Poset_Z$ we have $[\tau]_Z\le [\tau]\splitting_Z$ for any splitting $\splitting$. Call this a \emph{one-head splitting} if $F_1$ is a tree. A \emph{simple splitting} is one that is $\sim_Z$-equivalent to a direct sum of finitely many copies of $1$ together with a one-head splitting $\splitting$. Intuitively, a splitting is simple if only one ``head'' actually gets split. Note that every splitting is a product of simple splittings.
\end{definition}

\begin{definition}[Weakly elementary]
Write $\wedge$ for the tree with one caret. Write $\wedge^2$ for the tree obtained from $\wedge$ by adding a caret to its left leaf. Call a splitting \emph{weakly elementary} if it is $\sim_Z$-equivalent to a direct sum of finitely many splittings, each of which is of one of the following forms:
\[
[1,1,1],[\wedge,(a^k,1),1]\text{ for some } k\in\Z,\text{ or } [\wedge^2,1,1]\text{.}
\]
Let us abuse notation and write $1=[1,1,1]$, $\wedge(a^k,1)=[\wedge,(a^k,1),1]$, and $\wedge^2=[\wedge^2,1,1]$, so the weakly elementary splittings are those that are $\sim_Z$-equivalent to direct sums of copies of $1$, $\wedge(a^k,1)$ ($k\in\Z$), and $\wedge^2$.
\end{definition}

Note that the splitting $[1,b^k,1][\wedge,(1,c^{-k}),1]$ equals $[\wedge,(a^k,c^k),\wedge][\wedge,(1,c^{-k}),1] = \wedge(a^k,1)$, so $\wedge(a^k,1)$ really is a splitting. Also note that any weakly elementary simple splitting is $\sim_Z$-equivalent to a direct sum of finitely many copies of $1$ together with a single splitting of the form $1$, $\wedge(a^k,1)$ for some $k\in\Z$, or $\wedge^2$.

For $[\tau]_Z\in \Poset_Z$ and $\splitting$ a splitting such that the product $[\tau]\splitting$ makes sense, so $[\tau]_Z\le [\tau]\splitting_Z$, write $[\tau]_Z\trianglelefteq [\tau]\splitting_Z$ if $\splitting$ is weakly elementary. Call a simplex in $|\Poset_Z|$ \emph{weakly elementary} if it is of the form $x_0<\cdots<x_k$ with $x_i\trianglelefteq x_j$ for all $i<j$. The weakly elementary simplices form a subcomplex of $|\Poset_Z|$. We will denote this subcomplex by $\Stein_2(B_*\wr\brGrig)_Z$, or $\Stein_Z$ for short, and call it the \emph{$Z$-Stein--Farley complex} for the braided R\"over group. (Later we will identify $\Stein_Z$ with a polysimplicial complex that is more akin to the Stein--Farley cube complexes constructed earlier for general $d$-ary cloning systems.)

\begin{proposition}\label{prop:roever_Stein_cible}
The $Z$-Stein--Farley complex $\Stein_Z$ is contractible.
\end{proposition}

\begin{proof}
We can follow the argument in Subsection~\ref{ssec:stein_farley} that the usual Stein--Farley complex for a $d$-ary cloning system is contractible, in exactly the same way, and reduce the problem to proving that the realization $|(x,y)|$ is contractible, for any $x\le y$ in $\Poset_Z$ with $x\not\trianglelefteq y$. This in turn will follow by the same argument as in Subsection~\ref{ssec:stein_farley}, provided we can show that for any $x\le y$ with $x\not\trianglelefteq y$ there exists a unique largest $y_0\in (x,y)$ such that $x\trianglelefteq y_0$. Without loss of generality $x=1$. Analogously to the proofs of Lemma~4.7 and Proposition~4.8 of \cite{belk16} (for the R\"over group), it suffices to prove that if we have both $\wedge(a^k,1)_Z\le y$ and $\wedge(a^\ell,1)_Z\le y$ for some $k\ne \ell$ then $\wedge^2_Z\le y$. Let $\splitting$ and $\splitting'$ be splittings such that $\wedge(a^k,1)\splitting_Z=y$ and $\wedge(a^\ell,1)\splitting'_Z=y$, so $\wedge(a^k,1)\splitting_Z=\wedge(a^\ell,1)\splitting'_Z$. Thus $\splitting^{-1}(a^{\ell-k},1)\splitting'\in B_m\wr Z$ for appropriate $m$ (identifying $B_m\wr Z$ with its image in $\Groupoid$ under $g\mapsto [1,g,1]$). Now note that if $\wedge^2_Z\not\le y$ then we have $\splitting=z\oplus \overline{\splitting}$ and $\splitting'=z'\oplus \overline{\splitting}'$ for some $z,z'\in Z$ and some splittings $\overline{\splitting}$ and $\overline{\splitting}'$ (that is, neither $\splitting$ nor $\splitting'$ can involve ``splitting the first head''). But then $\splitting^{-1}(a^{\ell-k},1)\splitting'$ is the direct sum of $z^{-1}a^{\ell-k}z'$ with some element of $B_{m-1}\wr Z$, which implies $z^{-1}a^{\ell-k}z'\in Z$, a contradiction since $a^{\ell-k}\not\in Z$. We conclude that $\wedge^2_Z\le y$.
\end{proof}

\subsection{Polysimplicial structure}\label{ssec:poly}

Now let us describe the polysimplicial structure on $\Stein_Z$. A \emph{polysimplex} is a euclidean polytope that is a product of simplices. A \emph{polysimplicial complex} is an affine cell complex whose cells are polysimplices, and such that any two cells intersect in a (possibly empty) common face of each. Polysimplicial complexes are a simultaneous generalization of simplicial and cubical complexes. The polysimplicial structure on $\Stein_Z$ is built out of the following pieces:

\begin{definition}[Basic polysimplex, bottom vertex]
Let $[\tau]_Z$ be a vertex in $\Stein_Z$, say with $h([\tau]_Z)=m$ (here $h$ is the ``number of feet'' function from before). For each $1\le i\le m$ let $S_i$ be one of the following sets:
\[
\{1\}\text{, } \{1,\wedge(a^k,1)\}\text{ for some }k\in\Z\text{, } \{1,\wedge^2\}\text{, or }\{1,\wedge(a^k,1),\wedge^2\}\text{ for some }k\in\Z\text{.}
\]
Then the corresponding \emph{basic polysimplex} $\poly([\tau];S_1,\dots,S_m)$ is the full subcomplex of $\Stein_Z$ spanned by the vertex set
\[
\{[\tau](\splitting_1\oplus\cdots\oplus\splitting_m)_Z\mid \splitting_i\in S_i\}\text{.}
\]
The \emph{bottom} vertex of this basic polysimplex is $[\tau]_Z$.
\end{definition}

Note that $\poly([\tau];S_1,\dots,S_m)$ is a subdivision of a product of simplices, one for each $S_i$, where the dimension of the simplex corresponding to $S_i$ is $|S_i|-1$. In particular it makes sense to view this as a polysimplex.

It is clear that every weakly elementary simplex lies in some basic polysimplex, so the basic polysimplices cover $\Stein_Z$. To see that the basic polysimplices form a polysimplicial complex with $\Stein_Z$ as a simplicial subdivision, it remains to show that any intersection of two basic polysimplices is a common face of each. Before proving this, we need a definition.

\begin{definition}[Depth, non-expanding]
For a forest $F$, the \emph{depth} of a leaf is the distance from the leaf to the root of the tree in $F$ containing the leaf. Let $[F_-,\beta(g_1,\dots,g_m),F_+]\in\Groupoid$. Call this groupoid element \emph{non-expanding at the $j$th foot} if for every leaf of the $j$th tree of $F_+$, say it is the $i$th leaf of $F_+$, we have that the depth of the $i$th leaf in $F_+$ is at most the depth of the $\pi(\beta)(i)$th leaf of $F_-$. (Note that this property is invariant under expansions and reductions, so is well defined on elements of $\Groupoid$.) If an element is non-expanding at the $j$th foot for all $j$, call it \emph{non-expanding}.
\end{definition}

The terminology ``non-expanding'' is inspired by the non-braided case, as in \cite{belk16}, where groupoid elements are homeomorphisms between disjoint unions of copies of the Cantor set.

\begin{lemma}\label{lem:poly_int}
Any intersection of two basic polysimplices is a common face of each.
\end{lemma}

\begin{proof}
The proof of Lemma~5.3 of \cite{belk16} for the R\"over group can almost be copied verbatim for the braided case (up to changing notation and some left/right conventions), with just one step requiring justification, which we will point out when we reach it. Let $P=\poly([\tau];S_1,\dots,S_m)$ and $Q=\poly([\omega];T_1,\dots,T_q)$ be basic polysimplices. Each $S_i$ has a natural total order, consistent with $\le$ in $\Poset_Z$ when considered up to the ``mod $Z$'' equivalence relation, so it makes sense to take the minimum of a collection of elements of a given $S_i$. Define a binary operation $\wedge_P$ on the vertices of $P$ via
\[
[\tau](\splitting_1\oplus\cdots\oplus\splitting_m)_Z \wedge_P [\tau](\splitting_1'\oplus\cdots\oplus\splitting_m')_Z \defeq [\tau](\min(\splitting_1,\splitting_1')\oplus\cdots\oplus\min(\splitting_m,\splitting_m'))_Z \text{.}
\]
Define $\wedge_Q$ analogously. If $P\cap Q=\emptyset$ we are done, so suppose not. Let $v,v'\in P\cap Q$, and we claim $v\wedge_P v'=v\wedge_Q v'$. Without loss of generality $v\wedge_P v'=[\tau]_Z$ and $v\wedge_Q v'=[\omega]_Z$. Choose $\splitting_i,\splitting_i'\in S_i$ and $\altsplitting_i,\altsplitting_i'\in T_i$ such that
\begin{align*}
&v=[\tau](\splitting_1\oplus\cdots\oplus\splitting_m)_Z=[\omega](\altsplitting_1\oplus\cdots\oplus\altsplitting_q)_Z\\
\text{ and } &v'=[\tau](\splitting_1'\oplus\cdots\oplus\splitting_m')_Z=[\omega](\altsplitting_1'\oplus\cdots\oplus\altsplitting_q')_Z \text{.}
\end{align*}
Now for appropriate $r$ and $p$, choose $\beta\in B_r$, $\beta'\in B_p$, and $z_1,\dots,z_r,z_1',\dots,z_p'\in Z$ such that
\begin{align*}
&[\tau](\splitting_1\oplus\cdots\oplus\splitting_m)=[\omega](\altsplitting_1\oplus\cdots\oplus\altsplitting_q)\beta(z_1,\dots,z_r)\\
\text{ and } &[\tau](\splitting_1'\oplus\cdots\oplus\splitting_m')=[\omega](\altsplitting_1'\oplus\cdots\oplus\altsplitting_q')\beta(z_1',\dots,z_p') \text{.}
\end{align*}
(As before we identify $B_r\wr Z$ and $B_p\wr Z$ with their images in $\Groupoid$ under $g\mapsto [1,g,1]$.) Solving for $[\omega]^{-1}[\tau]$ in each equation yields
\begin{align*}
[\omega]^{-1}[\tau] &= (\altsplitting_1\oplus\cdots\oplus\altsplitting_q)\beta(z_1,\dots,z_r)(\splitting_1\oplus\cdots\oplus\splitting_m)^{-1} \\
&=(\altsplitting_1'\oplus\cdots\oplus\altsplitting_q')\beta'(z_1',\dots,z_p')(\splitting_1'\oplus\cdots\oplus\splitting_m')^{-1} \text{.}
\end{align*}

We want to conclude that $[\tau]_Z=[\omega]_Z$, and this is the one step that cannot be copied verbatim from the non-braided case. However, it does work analogously, using the above definition of non-expanding. Since $v\wedge_P v'=[\tau]_Z$, we either have $\splitting_i=1$ or $\splitting_i'=1$ for each $i$. This shows that $[\omega]^{-1}[\tau]$ is non-expanding. Analogously, $[\tau]^{-1}[\omega]$ is non-expanding. The only way this can happen, given the options for the $\splitting_i$ and $\altsplitting_i$, is if every $\splitting_i$ and $\altsplitting_i$ is trivial. We conclude that $[\tau]_Z=[\omega]_Z$, so $\wedge_P$ and $\wedge_Q$ agree when restricted to $P\cap Q$. By the exact same argument as in the proof of \cite[Lemma~5.3]{belk16}, this shows that without loss of generality $[\tau]=[\omega]$, and $P\cap Q=\poly([\tau];S_1\cap T_1,\dots,S_m\cap T_m)$, which is a common face of both $P$ and $Q$.
\end{proof}

At this point we can view $\Stein_Z$ either with its simplicial structure, or with this new polysimplicial structure. We will write $\Stein_Z$ for both, and no confusion should arise.

\begin{lemma}\label{lem:br_ro_stabs}
Any stabilizer in $\brR$ of a polysimplex in $\Stein_Z$ is of type $\F_\infty$.
\end{lemma}

\begin{proof}
First note that the stabilizer of any vertex in $\Stein_Z$ with $h$ value $m$ is isomorphic to $B_m\wr Z$, by the same argument as in Subsection~\ref{ssec:stein_farley} for $\Stein_d(G_*)$. Since $\Stein_Z$ is not upward-locally finite (thanks to the weakly elementary splittings $\wedge(a^k,1)$ for arbitrary $k\in\Z$), a polysimplex stabilizer need not have finite index in a vertex stabilizer, so we need to be more careful now. Let $P=\poly([\tau];S_1,\dots,S_m)$ be a polysimplex. Write $\Stab_{\brR}^0(P)$ for the pointwise stabilizer of $P$, so $\Stab_{\brR}^0(P)$ is a finite index subgroup of $\Stab_{\brR}(P)$. Since the bottom vertex $[\tau]_Z$ is the only vertex of $P$ with $m$ feet, any element stabilizing $P$ must fix $[\tau]_Z$. Thus we have inclusions $\Stab_{\brR}^0(P)\le \Stab_{\brR}(P)\le \Stab_{\brR}([\tau]_Z)$.

Note that $B_m\wr Z\cong \Stab_{\brR}([\tau]_Z)$, and analogously to the argument in Subsection~\ref{ssec:stein_farley}, this isomorphism can be realized as
\[
\eta\colon \beta(z_1,\dots,z_m)\mapsto [\tau][1,\beta(z_1,\dots,z_m),1][\tau]^{-1}\text{.}
\]
If $\beta$ is pure, then for any $\splitting_i\in S_i$ we have $[1,\beta,1](\splitting_1\oplus\cdots\oplus\splitting_m)_Z = (\splitting_1\oplus\cdots\oplus\splitting_m)_Z$. Hence the restriction of $\eta$ to $PB_m$ lands in $\Stab_{\brR}^0(P)$. The standard projection $B_m\wr Z \to B_m$ induces, via $\eta$, an epimorphism $\theta\colon \Stab_{\brR}([\tau]_Z)\to B_m$. We have shown that the image of $\Stab_{\brR}^0(P)$ under $\theta$ contains $PB_m$, and so has finite index in $B_m$. The kernel of $\theta$ is isomorphic to $Z^m$. Since any finite index subgroup of $B_m$ is of type $\F_\infty$, and since any subgroup of $Z^m$ is finitely generated free abelian, hence type $\F_\infty$, we conclude that $\Stab_{\brR}^0(P)$ is of type $\F_\infty$, being an extension of a type $\F_\infty$ group by a type $\F_\infty$ group. Since $\Stab_{\brR}^0(P)$ has finite index in $\Stab_{\brR}(P)$, we are done.
\end{proof}

The failure of $\Stein_Z$ to be upward-locally finite not only makes the stabilizer argument harder, but also makes the cocompactness argument harder. Let $\Stein_Z^{h\le m}$ be the full subcomplex of $\Stein_Z$ spanned by all vertices with at most $m$ feet.

\begin{lemma}\label{lem:br_ro_cocpt}
For each $m\in\N$, the sublevel complex $\Stein_Z^{h\le m}$ is cocompact under the action of $\brR$.
\end{lemma}

\begin{proof}
By the same argument as in Subsection~\ref{ssec:stein_farley} for $\Stein_d(G_*)$, the action is transitive on vertices with a given number of feet, so in $\Stein_Z^{h\le m}$ there are finitely many (in fact exactly $m$) vertex orbits. However, since we no longer have upward-local finiteness, we need to do more work to see that there are finitely many polysimplex orbits. It suffices to prove that for any vertex $[\tau]_Z$, there are finitely many $\Stab_{\brR}([\tau]_Z)$-orbits of polysimplices with $[\tau]_Z$ as their bottom vertex. Let $P=\poly([\tau];S_1,\dots,S_m)$ be a polysimplex with $[\tau]_Z$ as its bottom vertex. As in the proof of Lemma~\ref{lem:br_ro_stabs}, we have $B_m\wr Z\cong \Stab_{\brR}([\tau]_Z)$ via the isomorphism
\[
\psi\colon \beta(z_1,\dots,z_m)\mapsto [\tau][1,\beta(z_1,\dots,z_m),1][\tau]^{-1}\text{.}
\]
Note that for any $k\in\Z$ we have
\begin{align*}
[1,b^{-k},1][\wedge,(a^k,1),1]_Z & =[\wedge,(a^{-k},c^{-k}),\wedge][\wedge,(a^k,1),1]_Z\\
&=[\wedge,(a^{-k},c^{-k})(a^k,1),1]_Z\\
&=[\wedge,(1,c^{-k}),1]_Z\\
&=[\wedge,(1,1),1]_Z\text{.}
\end{align*}
Moreover,
\begin{align*}
[1,b^{-k},1][\wedge^2,1,1]_Z &=[\wedge^2,(\zeta^{-k}(1,1),c^{-k}),\wedge^2][\wedge^2,1,1]_Z\\
&=[\wedge^2,(\zeta^{-k}(1,1),c^{-k}),1]_Z\\
&=[\wedge^2,1,1]_Z\text{.}
\end{align*}
(Recall that $\zeta$ is our chosen generator of $B_2\cong\Z$, with $a=\zeta(1,1)$.) Define the \emph{signature} $||S_i||$ of $S_i$ to be $1$ if $S_i=\{1\}$, $2$ if $S_i=\{1,\wedge(a^k,1)\}$ for some $k\in\Z$, $3$ if $S_i=\{1,\wedge^2\}$, and $4$ if $S_i=\{1,\wedge(a^k,1),\wedge^2\}$ for some $k\in\Z$. Extrapolating using direct sums, the above calculations show that the $\Stab_{\brR}([\tau]_Z)$-orbit of $P$ includes every $\poly([\tau];T_1,\dots,T_m)$ with $||T_i||=||S_i||$ for each $1\le i\le m$. Since there are only finitely many (four) possibilities for each $||S_i||$, we are done.
\end{proof}

Combining Proposition~\ref{prop:roever_Stein_cible}, Lemma~\ref{lem:br_ro_stabs}, and Lemma~\ref{lem:br_ro_cocpt}, we get the following:

\begin{corollary}\label{cor:br_ro_action}
The group $\brR$ acts on the contractible complex $\Stein_Z$ with cell stabilizers of type $\F_\infty$, and with cocompact sublevel complexes $\Stein_Z^{h\le m}$.\qed
\end{corollary}

\subsection{Descending links}\label{ssec:br_ro_dlks}

By Citation~\ref{cit:brown_morse} and Corollary~\ref{cor:br_ro_action}, to get $\brR$ to be of type $\F_\infty$ it remains to prove that descending links get arbitrarily highly connected. (Note that $h$ is affine on polysimplices, so it really is a Morse function.) Much of the work in this subsection will be inspired by \cite[Section~6]{belk16} and \cite[Subsection~4.3]{skipper21}.

Since $\Stein_Z$ is a polysimplicial complex, the link of a vertex $v$ is a simplicial complex, with a $k$-simplex for each $(k+1)$-dimensional polysimplex containing $v$. If the polysimplex achieves its maximum $h$ value at $v$, then the corresponding simplex in the link of $v$ is even in the descending link. Note that vertices with the same number of feet have isomorphic descending links, so just like the complexes $\dlmodel_d^m(B_*\wr G)$ from Subsection~\ref{ssec:braid_dlks}, we can denote by $\dlmodel^m_Z$ the simplicial complex that is isomorphic to the descending link of any vertex with $m$ feet.

\begin{definition}[Merging]
Call an element $\merging$ of $\Groupoid$ a \emph{merging} if $\merging^{-1}$ is a splitting. Call it \emph{simple} or \emph{weakly elementary} if its inverse is, as a splitting. If $v=[\tau]_Z$ is a vertex in $\Stein_Z$ and $\merging$ is a merging such that the product $[\tau]\merging$ makes sense, then we call $[\tau]\merging_Z$ a \emph{merging of $v$}.
\end{definition}

For $\merging$ a merging, write $\merging^{(i)}$ for the merging $1\oplus\cdots\oplus 1\oplus \merging\oplus 1\oplus\cdots\oplus 1$, where $\merging$ is the $i$th direct summand. We see that the vertices of the descending link $\dlk v$ of $v$ are precisely the non-trivial, weakly elementary, simple mergings of $v$. Writing $\vee\defeq\wedge^{-1}$, this means that to obtain a vertex of $\dlk v$, we choose a representative $[\tau]$ of $v=[\tau]_Z$, take $\merging$ to be either $(a^k,1)\vee$ for some $k\in\Z$ or $\vee^2$, and then take $[\tau]\merging^{(i)}_Z$ for some $i$. If we already have a fixed representative $[\tau]$ of $v$, then in addition to all of the above we first multiply $[\tau]$ on the right by an element of $B_m\wr Z$, where $m=h(v)$ (to account for potentially changing representatives). In all that follows, we will often identify $B_m\wr Z$ with its image in $\Groupoid$ under $g\mapsto [1,g,1]$.

We can denote vertices of $\dlmodel^m_Z$ by $g\merging^{(i)}_Z$ for $g\in B_m\wr Z$, $\merging$ equal to either $(a^k,1)\vee$ for some $k\in\Z$ or $\vee^2$, and $i$ some appropriate number. If $\dlmodel^m_Z$ is representing $\dlk [\tau]_Z$, then $g\merging^{(i)}_Z$ represents $[\tau]g\merging^{(i)}_Z$.

\begin{definition}[Type, mass]
Let $[\tau]_Z$ be a vertex of $\Stein_Z$ with $h([\tau]_Z)=m$. Consider a vertex of $\dlk [\tau]_Z$, which is a non-trivial, weakly elementary, simple merging of $[\tau]_Z$, so one of the form $[\tau]g\merging^{(i)}_Z$, for $g=\beta(z_1,\dots,z_m)\in B_m\wr Z$, $\merging$ of the form $(a^k,1)\vee$ for some $k\in\Z$ or $\vee^2$, and some $i$. Define the \emph{type} of this merging to be $\merging$, so the type is either $(a^k,1)\vee$ for some $k\in\Z$ or $\vee^2$. Define the \emph{mass} of this merging to be $2$ if it is of type $(a^k,1)\vee$ for some $k\in\Z$, and $3$ if it is of type $\vee^2$ (so the mass is the ``number of feet'' that $\merging$ is merging together).
\end{definition}

In order to understand $\dlmodel^m_Z$, we will map it to a certain complex of disks, defined as follows.

\begin{definition}[$(2,5/2)$-disk complex]\label{def:252}
Let $\surface_{b,m}^g$ be a surface with $b$ boundary components, $m$ marked points, and genus $g$. The \emph{$(2,5/2)$-marked-point-disk complex} $\diskcpx_{2,5/2}(\surface_{b,m}^g)$ is the simplicial complex defined as follows. A vertex is an isotopy class of an embedded disk enclosing precisely $2$ marked points in its interior and either $0$ or $1$ marked points in its boundary. If a disk has $0$ marked points in its boundary call it a \emph{$2$-disk} and if it has $1$ marked point in its boundary call it a \emph{$5/2$-disk} (two points inside plus one point ``halfway inside'' equals two and a half points inside, hence the name $5/2$-disk). A $k$-simplex in $\diskcpx_{2,5/2}(\surface_{b,m}^g)$ is an isotopy class of $k+1$ (pairwise non-isotopic) embedded disks such that each one is either a $2$-disk or a $5/2$-disk, and any two of them are either disjoint or nested. The face relation is given by taking subsets.
\end{definition}

Here when we call two isotopy classes ``disjoint'' or ``nested'' we of course mean that they admit disjoint or nested representatives. If one disk is nested in a (non-isotopic) second disk then necessarily the former is a $2$-disk and the latter is a $5/2$-disk. Similar to how the usual disk complex can be viewed as a subcomplex of the curve complex (outside some pathological cases), this $(2,5/2)$-marked-point-disk complex can usually be viewed as a subcomplex of the arc-and-curve complex: the boundary of a $2$-disk is a curve and the boundary of a $5/2$-disk is an arc. See Figure~\ref{fig:nested_disks} for an example. (The astute reader may wonder why we are using $5/2$-disks instead of $3$-disks; the answer is that the analog of Proposition~\ref{prop:br_ro_complete_join} would not hold if we used $3$-disks.)

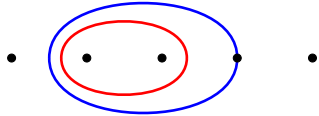
\begin{figure}[htb]
\centering
\begin{tikzpicture}
  \draw[red,line width=1pt]
  (0.66,0) to[out=90,in=90] (2.33,0)
  (0.66,0) to[out=-90,in=-90] (2.33,0);
  
  \draw[blue,line width=1pt]
  (0.5,0) to[out=90,in=90] (3,0)
  (0.5,0) to[out=-90,in=-90] (3,0);
  
  \filldraw
   (0,0) circle (1.5pt)
   (1,0) circle (1.5pt)
   (2,0) circle (1.5pt)
   (3,0) circle (1.5pt)
   (4,0) circle (1.5pt);
\end{tikzpicture}
\caption{An example of a $1$-simplex in $\diskcpx_{2,5/2}(\surface_{1,5}^0)$. The disk bounded in red is a $2$-disk, the disk bounded in blue is a $5/2$-disk, and they are nested.}
\label{fig:nested_disks}
\end{figure}

Now let us map $\dlmodel^m_Z$ to $\diskcpx_{2,5/2}(\surface_m)$, where $\surface_m=\surface_{1,m}^0$. To start, we will just define the map on the vertices. A vertex of $\dlmodel^m_Z$ is of the form $\beta(z_1,\dots,z_m)\merging^{(i)}_Z$, where $\beta(z_1,\dots,z_m)\in B_m\wr Z$, $\merging$ is of the form $(a^k,1)\vee$ for some $k\in\Z$ or $\vee^2$, and $i$ is some number. Define
\[
\mu\colon (\dlmodel^m_Z)^{(0)} \to \diskcpx_{2,5/2}(\surface_m)^{(0)}
\]
as follows. If $\merging=(a^k,1)\vee$ for some $k$ then define $\disk(\merging^{(i)})$ to be a $2$-disk that is a tubular neighborhood of the $2$-matching $\{i,i+1\}$ (recall from before that we fix an embedding of the linear graph into the surface, and so matchings make sense). If $\merging=\vee^2$ then define $\disk(\merging^{(i)})$ to be a $5/2$-disk that is obtained by taking a tubular neighborhood of the subset obtained from the $3$-matching $\{i,i+1,i+2\}$ by removing a small disk centered at $i+2$, in such a way that $i+2$ lies precisely on the boundary of the neighborhood. For example, in Figure~\ref{fig:nested_disks} if we label the marked points $1$ through $5$ from left to right, then (up to isotopy) the disk bounded in blue is $\disk\left((\vee^2)^{(2)}\right)$. Now define $\mu$ via
\[
\mu \colon \beta(z_1,\dots,z_m)\merging^{(i)}_Z \mapsto (\disk(\merging^{(i)}))\beta^{-1} \text{.}
\]

\begin{lemma}\label{lem:br_ro_well_def}
The map $\mu$ is well defined on vertices.
\end{lemma}

\begin{proof}
The proof is similar to that of Lemma~\ref{lem:well_def}. Let $\gamma(y_1,\dots,y_q)\in B_q\wr Z$ for appropriate $q$, so
\[
\beta(z_1,\dots,z_m)\merging^{(i)}_Z = \beta(z_1,\dots,z_m)\merging^{(i)}\gamma(y_1,\dots,y_q)_Z \text{.}
\]
Let $\gamma'(y_1',\dots,y_m')\in B_m\wr \brGrig$ and $i'$ be such that
\[
\merging^{(i)}\gamma(y_1,\dots,y_q) = \gamma'(y_1',\dots,y_m')\merging^{(i')} \text{.}
\]
Note that the $y_j'$ might not lie in $Z$, but the only way this could happen is when $j=i'$, and then only if $\merging$ has mass $2$ so $y_{i'}'$ is a power of $a$. In this case, we can let $\merging'$ be $\merging=(a^k,1)\vee$ with $k$ changed appropriately, and (changing the definition of $y_{i'}'$) get
\[
\merging^{(i)}\gamma(y_1,\dots,y_q) = \gamma'(y_1',\dots,y_m')(\merging')^{(i')}
\]
with all $y_j'$ in $Z$. Now we have
\[
\beta(z_1,\dots,z_m)\merging^{(i)}\gamma(y_1,\dots,y_q)_Z = \beta\gamma'(z_1',\dots,z_m')(y_1',\dots,y_m')(\merging')^{(i')}
\]
for some $z_i'\in Z$. Since all the $z_j$ and $y_j'$ are in $Z$, this maps under $\mu$ to $(\disk((\merging')^{(i')}))(\beta\gamma')^{-1}$. We need to show this equals $(\disk(\merging^{(i)}))\beta^{-1}$, or equivalently that $(\disk(\merging^{(i)}))\gamma' = \disk((\merging')^{(i')})$. If $\merging$ (and hence $\merging'$) has mass $2$ then this follows by the same argument as in the proof of Lemma~\ref{lem:well_def}. Now suppose $\merging=\merging'=\vee^2$. Then $\gamma'$ is obtained from $\gamma$ by replacing the $i$th strand (counting from the top) with $3$ parallel strands -- call this intermediate braid $\gamma''$ -- and possibly braiding the first two of these around each other some number of times (depending on the braided wreath recursion of $y_{i'}$, which has some power of $a$ in its left entry and some element of $Z$ in its right entry). When acting on $\disk(\merging^{(i)})$, this braiding of the first $2$ strands does not affect the (equivalence class of the) $5/2$-disk, so $(\disk(\merging^{(i)}))\gamma' = (\disk(\merging^{(i)}))\gamma''$. It is clear that $(\disk(\merging^{(i)}))\gamma'' = \disk((\merging')^{(i')})$, so we are done.
\end{proof}

Now we want to extend $\mu$ to all simplices of $\dlmodel^m_Z$. We need to show that if vertices $w_0,\dots,w_k$ span a simplex in $\dlmodel^m_Z$ then their images under $\mu$ span a simplex in $\diskcpx_{2,5/2}(\surface_m)$. Since $\diskcpx_{2,5/2}(\surface_m)$ is a flag complex by construction, it suffices to do this for $k=1$, i.e., to show the following:

\begin{lemma}[Extends to simplices]\label{lem:span_edge}
If two vertices $w$ and $u$ span an edge in $\dlmodel^m_Z$ then their images under $\mu$ span an edge in $\diskcpx_{2,5/2}(\surface_m)$.
\end{lemma}

\begin{proof}
View $\dlmodel^m_Z$ as $\dlk v$ for some vertex $v=[\tau]_Z$ of $\Stein_Z$ with $h(v)=m$, so $w$ and $u$ are adjacent vertices to $v$ with $h(w),h(u)<m$. The fact that $w$ and $u$ span an edge in $\dlk v$ means that there is a $2$-dimensional polysimplex, i.e., a square or a triangle, in $\Stein_Z$ containing $v$, $w$, and $u$. First suppose it is a square. Then there is a merging of $v$ of the form $[\tau]\beta(z_1,\dots,z_m)\merging_Z$ for $\beta(z_1,\dots,z_m)\in B_m\wr Z$ and
\[
\merging=1\oplus\cdots\oplus 1\oplus \merging_1\oplus 1\oplus\cdots\oplus 1\oplus \merging_2\oplus 1 \text{,}
\]
say with $\merging_1$ in the $i_1$th spot and $\merging_2$ in the $i_2$th spot, such that the following holds: If $\altmerging_i$ is $\merging$ with $\merging_i$ replaced by an appropriate number of copies of $1$, then $w=[\tau]\beta(z_1,\dots,z_m)(\altmerging_1)_Z$ and $u=[\tau]\beta(z_1,\dots,z_m)(\altmerging_2)_Z$. Clearly $\disk(\altmerging_1)$ and $\disk(\altmerging_2)$ are disjoint, hence so are $\mu(w)=(\disk(\altmerging_1))\beta^{-1}$ and $\mu(u)=(\disk(\altmerging_2))\beta^{-1}$. Now suppose the $2$-dimensional polysimplex containing $v$, $w$, and $u$ is a triangle. This means there we can choose $[\tau']$, $k$, and $i$ such that $v$, $w$, and $u$ are the vertices of the polysimplex $\poly([\tau'];S_1,\dots,S_{m-2})$, where $S_j=\{1\}$ for all $j\ne i$ and $S_i=\{1,\wedge(a^k,1),\wedge^2\}$. Say $w$ has mass $3$ relative to $v$ (and $u$ has mass $2$), so $w=[\tau']_Z$, $u=[\tau'](\wedge(a^k,1))^{(i)}_Z$, and $v=[\tau'](\wedge^2)^{(i)}_Z$. Without loss of generality $[\tau]=[\tau'](\wedge^2)^{(i)}$, so $u=[\tau](\vee a^k)^{(i)}_Z=[\tau](\zeta^k \vee)^{(i)}_Z$ and $w=[\tau](\vee^2)^{(i)}_Z$. Now $\mu(u)=\left(\disk(\vee^{(i)})\right)\left((\zeta^k)^{(i)}\right)^{-1}=\disk(\vee^{(i)})$ and $\mu(w)=\disk\left((\vee^2)^{(i)}\right)$, and the former is nested in the latter as desired.
\end{proof}

At this point we have a simplicial map
\[
\mu\colon \dlmodel^m_Z \to \diskcpx_{2,5/2}(\surface_m) \text{.}
\]
Our next goal is to prove that it is a surjective complete join. It is clearly surjective on vertices, and is simplexwise injective, so it suffices to show the following:

\begin{proposition}\label{prop:br_ro_complete_join}
Let $w_0,\dots,w_k$ be vertices of $\dlmodel^m_Z$ whose images under $\mu$ span a $k$-simplex in $\diskcpx_{2,5/2}(\surface_m)$. Then $w_0,\dots,w_k$ span a $k$-simplex in $\dlmodel^m_Z$. Hence $\mu$ is a complete join.
\end{proposition}

\begin{proof}
If $k=0$ then there is nothing to show, so assume $k>0$. By induction, $w_1,\dots,w_k$ span a simplex in $\dlmodel^m_Z$, call it $\sigma$. The image $\mu(\sigma)$ is represented by $k$ disks that are pairwise disjoint or nested, and such that the disk $\mu(w_0)$ is disjoint or nested with each of them. If $\mu(w_0)$ is disjoint from all the disks in $\mu(\sigma)$, then the argument in the proof of Proposition~\ref{prop:dlk_conn} shows that $w_0,\dots,w_k$ span a simplex in $\dlmodel^m_Z$. Now assume this is not the case, so without loss of generality $w_0$ has mass $2$, $w_1$ has mass $3$, $\mu(w_0)$ is nested in $\mu(w_1)$, and both $\mu(w_0)$ and $\mu(w_1)$ are disjoint from $\mu(w_i)$ for all $2\le i\le k$. Now view $\dlmodel^m_Z$ as $\dlk v$ for a vertex $v=[\tau]_Z$ with $h(v)=m$, so $\sigma$ is represented by a polysimplex containing $v,w_1,\dots,w_k$, with $v$ as its vertex maximizing $h$. Write this polysimplex as $\poly([\tau'];S_1,\dots,S_r)$, so $v=[\tau'](\max(S_1)\oplus\cdots\oplus\max(S_r))_Z$. We can assume $[\tau]=[\tau'](\max(S_1)\oplus\cdots\oplus\max(S_r))$. Without loss of generality $S_1=\{1,\wedge^2\}$ and
\[
w_1=[\tau'](1\oplus\max(S_2)\oplus\cdots\oplus\max(S_r))_Z = [\tau](\vee^2)^{(1)}_Z \text{.}
\]
Now say $w_0=[\tau]\beta(z_1,\dots,z_m)((a^\ell,1)\vee)^{(i)}_Z$ for some $\beta(z_1,\dots,z_m)\in B_m\wr Z$, some $\ell$, and some $i$. Up to $\sim_Z$ we can actually assume $\ell=0$, e.g, right multiply by $(b^{-\ell})^{(i)}$. Then $\mu(w_0)=\left(\disk(\vee^{(i)})\right)\beta^{-1}$, and this is disjoint from $\mu(w_2),\dots,\mu(w_k)$ and contained in $\mu(w_1)$. By the above equation for $w_1$ we have $\mu(w_1)=\disk\left((\vee^2)^{(1)}\right)$, so $\mu(w_0)=\disk\left(\vee^{(1)}\right)$, hence $\disk\left(\vee^{(i)}\right) = \left(\disk\left(\vee^{(1)}\right)\right)\beta$. This shows that in fact $w_0=[\tau](\zeta^\ell \vee)^{(1)}_Z$ for some $\ell$, hence
\[
w_0=[\tau'](\max(S_1)\oplus\cdots\oplus\max(S_r))(\zeta^\ell \vee)^{(1)}_Z \text{.}
\]
Since $\max(S_1)=\wedge^2$ and $\wedge\zeta^\ell \vee=a^\ell$, we conclude that
\[
w_0=[\tau'](\wedge(a^\ell,1)\oplus\cdots\oplus\max(S_r))(\zeta^\ell \vee)^{(1)}_Z \text{.}
\]
Now setting $S_1'=\{1,\wedge(a^\ell,1),\wedge^2\}$, we see that the polysimplex $\poly([\tau'];S_1',S_2,\dots,S_r)$ contains $w_0$. Since it has $\poly([\tau'];S_1,\dots,S_r)$ as a face it also contains $v,w_1,\dots,w_k$, so we conclude that $w_0,\dots,w_k$ span a simplex in $\dlk v$.
\end{proof}

Now that we know $\mu$ is a complete join, we can understand higher connectivity of $\dlmodel^m_Z$ by analyzing higher connectivity of $\diskcpx_{2,5/2}(\surface_m)$.

\begin{proposition}\label{prop:2_3_disk_conn}
The complex $\diskcpx_{2,5/2}(\surface_{b,m}^g)$ is wCM of dimension $\lfloor\frac{m+1}{3}\rfloor-1$.
\end{proposition}

\begin{proof}
First we show that $\diskcpx_{2,5/2}(\surface_{b,m}^g)$ is homotopy equivalent to $\diskcpx_2(\surface_{b,m}^g)$. Given a $k$-simplex $\sigma=\{D_0,\dots,D_k\}$ in $\diskcpx_{2,5/2}(\surface_{b,m}^g)$, let $n_{5/2}(\sigma)$ be the number of $D_i$'s that are $5/2$-disks. Define a Morse function on the barycentric subdivision $\diskcpx_{2,5/2}(\surface_{b,m}^g)'$ by sending a $k$-simplex $\sigma$ to the ordered pair $(n_{5/2}(\sigma),-\dim(\sigma))$. This takes on finitely many values in $\R\times\R$, so viewed lexicographically the values are a finite totally ordered set, and adjacent (i.e., properly incident) simplices have different values, hence this really can be viewed as a Morse function. The sublevel complex of all $\sigma$ with $(n_{5/2}(\sigma),-\dim(\sigma))\le (0,0)$ equals $\diskcpx_2(\surface_{b,m}^g)$. In particular, if we can show that the descending link of any $\sigma$ with $(n_{5/2}(\sigma),-\dim(\sigma))>(0,0)$ is contractible, then Morse theory will tell us that the inclusion of $\diskcpx_2(\surface_{b,m}^g)$ into $\diskcpx_{2,5/2}(\surface_{b,m}^g)$ is a homotopy equivalence. (Intuitively, we attach the missing vertices in increasing order with respect to the Morse function, and at each stage we do so along a contractible descending link, so the homotopy type never changes.) To this end, suppose $(n_{5/2}(\sigma),-\dim(\sigma))>(0,0)$, and we need to prove the descending link of $\sigma$ is contractible. The descending link is a join, of the \emph{descending face link} spanned by all faces $\sigma^\vee<\sigma$ with $n_{5/2}(\sigma^\vee)<n_{5/2}(\sigma)$, and the \emph{descending coface link} spanned by all proper cofaces $\sigma^\wedge>\sigma$ with $n_{5/2}(\sigma^\wedge)=n_{5/2}(\sigma)$. For any $5/2$-disk $D$ there is a unique $2$-disk contained in it, which we will denote by $\delta(D)$. Let $\sigma=\{D_0,\dots,D_k\}$. First suppose that some $D_i$ is a $5/2$-disk such that $\delta(D_i)\not\in\sigma$. Then $\sigma\cup\{\delta(D_i)\}$ is a descending proper coface of $\sigma$, and moreover for any descending proper coface $\sigma^\wedge$ of $\sigma$ we have that $\sigma^\wedge \cup\{\delta(D_i)\}$ is also a descending proper coface. The map $\sigma^\wedge\mapsto \sigma^\wedge \cup\{\delta(D_i)\}$ thus induces a homotopy equivalence from the descending coface link to the (contractible) star of $\sigma\cup\{\delta(D_i)\}$, so the descending coface link is contractible (see, e.g., \cite[Section~1.5]{quillen78}). Now suppose whenever $D_i$ is a $5/2$-disk, $\delta(D_i)\in\sigma$. Let $\sigma_0<\sigma$ be the (non-empty) face of $\sigma$ consisting of all $2$-disks in $\sigma$. This is a descending face, and given any descending face $\sigma^\vee<\sigma$ we have that $\sigma^\vee \cup \sigma_0$ is again a descending face of $\sigma$. Hence the map $\sigma^\vee\mapsto \sigma^\vee \cup \sigma_0$ induces a homotopy equivalence from the descending face link to the (contractible) star of $\sigma_0$, so the descending face link is contractible.

We have shown that $\diskcpx_{2,5/2}(\surface_{b,m}^g)$ is homotopy equivalent to $\diskcpx_2(\surface_{b,m}^g)$, which is $(\left\lfloor\frac{m+1}{3}\right\rfloor-2)$-connected by Citation~\ref{cit:disk_conn}. Since higher connectivity of links is not necessarily preserved under homotopy equivalence, we have more work to do before we can conclude that $\diskcpx_{2,5/2}(\surface_{b,m}^g)$ is wCM of dimension $\left\lfloor\frac{m+1}{3}\right\rfloor-1$. Let $\sigma=\{D_0,\dots,D_k\}$ be a $k$-simplex in $\diskcpx_{2,5/2}(\surface_{b,m}^g)$, and consider its link $L\defeq \lk \sigma$ in $\diskcpx_{2,5/2}(\surface_{b,m}^g)$. Working with the barycentric subdivision as in the first paragraph, we can view $L$ as the subcomplex spanned by all proper cofaces $\tau$ of $\sigma$. If some $D_i$ is a $5/2$-disk such that $\delta(D_i)\not\in\sigma$, then the map $\tau\mapsto \tau \cup\{\delta(D_i)\}$ induces a homotopy equivalence from $L$ to the star of the proper coface $\sigma\cup\{\delta(D_i)\}$ in $L$, hence $L$ is contractible. Now assume whenever $D_i$ is a $5/2$-disk, $\delta(D_i)\in\sigma$. Call $D_i$ and $\delta(D_i)$ \emph{paired} in this case. If $D_i$ is a vertex of $\sigma$ that is not paired, call it \emph{solitary} (and note that the solitary vertices must be $2$-disks). Say $2p$ is the number of paired vertices in $\sigma$, so $k-2p+1$ is the number of solitary vertices. Consider the surface $\surface_{b+p,m-3p}^g$ obtained by cutting along the boundary and removing the interior (and un-marking the marked point on the boundary) of each $5/2$-disk in $\sigma$. Let $\overline{\sigma}$ be the (possibly empty) $(k-2p)$-simplex in $\diskcpx_{2,5/2}(\surface_{b+p,m-3p}^g)$ whose vertices are the solitary vertices of $\sigma$, now viewed as disks in $\surface_{b+p,m-3p}^g$. It is clear that $L$ is isomorphic to the link of $\overline{\sigma}$ in $\diskcpx_{2,5/2}(\surface_{b+p,m-3p}^g)$ (which if $\overline{\sigma}$ is empty means the whole complex). If the desired result holds for $\overline{\sigma}$ and $\surface_{b+p,m-3p}^g$, then the link of $\overline{\sigma}$ in $\diskcpx_{2,5/2}(\surface_{b+p,m-3p}^g)$ (and hence also $L$) is $\left(\left\lfloor\frac{(m-3p)+1}{3}\right\rfloor-(k-2p)-3\right)$-connected, hence $(\left\lfloor\frac{m+1}{3}\right\rfloor-k-3)$-connected, i.e., the desired result holds for $\sigma$ and $\surface_{b,m}^g$. Thus it remains only to prove the result in the case when all the vertices of $\sigma$ are solitary.

Since we have reduced to a case where all the vertices of $\sigma$ are $2$-disks, we can consider the subcomplex $L_2$ of $L$ that is the link of $\sigma$ in the subcomplex $\diskcpx_2(\surface_{b,m}^g)$ of $\diskcpx_{2,5/2}(\surface_{b,m}^g)$. Since $\diskcpx_2(\surface_{b,m}^g)$ is wCM of dimension $\left\lfloor\frac{m+1}{3}\right\rfloor-1$ by Corollary~\ref{cor:disk_conn}, we know $L_2$ is $(\left\lfloor\frac{m+1}{3}\right\rfloor-k-3)$-connected. Now we will build up from $L_2$ to $L$ by gluing in the missing simplices $\tau$. We can again use the Morse function $(n_{5/2},-\dim)$, now restricted to $L$, since $L_2$ is precisely the sublevel complex of $L$ defined by $(n_{5/2},-\dim)\le (0,0)$. If we can show that the descending link of any simplex in $L\setminus L_2$ is $(\left\lfloor\frac{m+1}{3}\right\rfloor-k-4)$-connected then discrete Morse theory (see, e.g., \cite[Corollary~2.6]{bestvina97}) will tell us that $L$ is $(\left\lfloor\frac{m+1}{3}\right\rfloor-k-3)$-connected and we will be done. The argument from the first paragraph works in exactly the same way to show that the descending link in $L$ of a simplex $\tau$ in $L\setminus L_2$ is contractible, except in one case, namely when every disk in $\tau\setminus\sigma$ is a $5/2$-disk that is paired with a $2$-disk in $\sigma$; see Figure~\ref{fig:link_nested_disks} for a visualization of this situation. In this case the descending face link of $\tau$ in $L$ is the entire boundary of $\tau$ (since removing any proper subset of disks in $\tau\setminus\sigma$ is a descending move), so if $\tau$ is an $\ell$-simplex in $L$ (i.e., a $(k+\ell+1)$-simplex in $\diskcpx_{2,5/2}(\surface_{b,m}^g)$ containing the $k$-simplex $\sigma$), then the descending face link of $\tau$ in $L$ is an $(\ell-1)$-sphere, hence $(\ell-2)$-connected. In this case we can also understand the descending coface link of $\tau$ in $L$: it is isomorphic to the $2$-disk complex on the surface obtained by cutting out all the disks in $\tau$. Since each disk in $\tau\setminus\sigma$ contains a disk in $\sigma$, this produces $k+1$ new boundary components and eliminates $2(k+1)+(\ell+1)=2k+\ell+3$ marked points, so this surface is $\surface_{b+k+1,m-2k-\ell-3}^g$. The $2$-disk complex of this surface, and hence the descending coface link of $\tau$, is $\left(\left\lfloor\frac{(m-2k-\ell-3)+1}{3}\right\rfloor-2\right)$-connected. Joining this with the descending face link, which is $(\ell-2)$-connected, we get that the descending link of $\tau$ is $\left(\left\lfloor\frac{(m-2k-\ell-3)+1}{3}\right\rfloor+\ell-2\right)$-connected. We need it to be $\left(\left\lfloor\frac{m+1}{3}\right\rfloor-k-4\right)$-connected, so it suffices to prove that $\left\lfloor\frac{m+1}{3}\right\rfloor-k-1 \le \left\lfloor\frac{m+1}{3} - \frac{2k+\ell}{3}\right\rfloor+\ell$. This is indeed true by the following calculation:
\begin{align*}
\left\lfloor\frac{m+1}{3}\right\rfloor-k-1 &= \left\lfloor\frac{m+1}{3} - \frac{3k+3}{3}\right\rfloor \\
&\le \left\lfloor\frac{m+1}{3} - \frac{2k-2\ell}{3}\right\rfloor \\
&\le \left\lfloor\frac{m+1}{3} - \frac{2k+\ell}{3}\right\rfloor+\ell
\end{align*}
\end{proof}

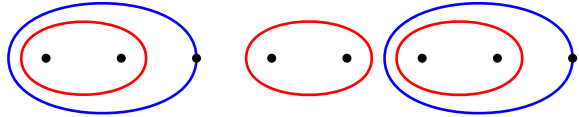
\begin{figure}[htb]
\centering
\begin{tikzpicture}
  \draw[red,line width=1pt]
  (-.33,0) to[out=90,in=90] (1.33,0)
  (-.33,0) to[out=-90,in=-90] (1.33,0);
  
  \draw[blue,line width=1pt]
  (-0.5,0) to[out=90,in=90] (2,0)
  (-0.5,0) to[out=-90,in=-90] (2,0);
  
  \draw[red,line width=1pt]
  (2.66,0) to[out=90,in=90] (4.33,0)
  (2.66,0) to[out=-90,in=-90] (4.33,0);
  
  \draw[red,line width=1pt]
  (4.66,0) to[out=90,in=90] (6.33,0)
  (4.66,0) to[out=-90,in=-90] (6.33,0);
  
  \draw[blue,line width=1pt]
  (4.5,0) to[out=90,in=90] (7,0)
  (4.5,0) to[out=-90,in=-90] (7,0);
  
  \filldraw
   (0,0) circle (1.5pt)
   (1,0) circle (1.5pt)
   (2,0) circle (1.5pt)
   (3,0) circle (1.5pt)
   (4,0) circle (1.5pt)
   (5,0) circle (1.5pt)
   (6,0) circle (1.5pt)
   (7,0) circle (1.5pt);
\end{tikzpicture}
\caption{An example of the last case in the proof of Proposition~\ref{prop:2_3_disk_conn} (for the surface $\surface_8$). The disks bounded in red comprise $\sigma$ and the disks bounded in blue comprise $\tau\setminus\sigma$. Hence $\sigma$ is a $2$-simplex and $\tau$ is a $4$-simplex in $\diskcpx_{2,5/2}(\surface_8)$ containing $\sigma$, so $\tau$ represents a $1$-simplex in $L=\lk\sigma$.}
\label{fig:link_nested_disks}
\end{figure}

Now we can prove the main result of this section, that $\brR$ is of type $\F_\infty$.

\begin{proof}[Proof of Theorem~\ref{thrm:braided_roever_F_infty}]
By Citation~\ref{cit:brown_morse} and Corollary~\ref{cor:br_ro_action}, it suffices to prove that the descending links of vertices in $\Stein_Z$ get arbitrarily highly connected as their number of feet goes to $\infty$. This follows from Citation~\ref{cit:join}, Proposition~\ref{prop:br_ro_complete_join}, and Proposition~\ref{prop:2_3_disk_conn}.
\end{proof}

\bibliographystyle{alpha}
\newcommand{\etalchar}[1]{$^{#1}$}

\end{document}